\documentclass{amsart}

\usepackage[T1]{fontenc}
\usepackage{enumerate, amsmath, amsfonts, amssymb, amsthm, mathrsfs, wasysym, graphics, graphicx, xcolor, url, hyperref, hypcap,  shuffle, xargs, multicol, overpic, pdflscape, multirow, hvfloat, minibox, accents, array, xifthen, a4wide, ae, aecompl}
\usepackage{marginnote}
\hypersetup{colorlinks=true, citecolor=darkblue, linkcolor=darkblue}
\usepackage[all]{xy}
\usepackage[bottom]{footmisc}
\usepackage{tikz}
\usepackage{tkz-graph}
\usetikzlibrary{trees, decorations, decorations.markings, shapes, arrows, matrix, calc, fit, intersections, patterns}
\graphicspath{{figures/}}
\usepackage{caption}
\title{Polytopal realizations of finite type $\b{g}$-vector fans}

\thanks{
	CH was supported by NSERC Discovery grant {\em Coxeter groups and related structures}.
	SS was a Marie Curie - Cofund Fellow at INdAM and is now supported by the ISF grant 1144/16.
	VP was partially supported by the French ANR grant SC3A~(15\,CE40\,0004\,01).
}

\author{Christophe Hohlweg}
\address[Christophe Hohlweg]{LaCIM, Universit\'e du Qu\'ebec \`A Montr\'eal (UQAM)}
\email{hohlweg.christophe@uqam.ca}
\urladdr{\url{http://hohlweg.math.uqam.ca/}}

\author{Vincent Pilaud}
\address[Vincent Pilaud]{CNRS \& LIX, \'Ecole Polytechnique, Palaiseau}
\email{vincent.pilaud@lix.polytechnique.fr}
\urladdr{\url{http://www.lix.polytechnique.fr/~pilaud/}}

\author{Salvatore Stella}
\address[Salvatore Stella]{University of Haifa}
\email{stella@math.haifa.ac.il}
\urladdr{\url{http://math.haifa.ac.il/~stella/}}

\newtheorem{theorem}{Theorem}%[section]
\newtheorem{corollary}[theorem]{Corollary}
\newtheorem{proposition}[theorem]{Proposition}
\newtheorem{lemma}[theorem]{Lemma}

\theoremstyle{definition}
\newtheorem{definition}[theorem]{Definition}
\newtheorem{example}[theorem]{Example}
\newtheorem{remark}[theorem]{Remark}

\newcommand{\R}{\mathbb{R}} % reals
\newcommand{\Q}{\mathbb{Q}} % rationals
\newcommand{\Z}{\mathbb{Z}} % integers
\newcommand{\cH}{\mathcal{H}} % hyperplane arrangement
\newcommand{\cX}{\mathcal{X}} % vectors
\renewcommand{\b}[1]{\mathbf{#1}} % bold letters
\newcommand{\set}[2]{\left\{ #1 \;\middle|\; #2 \right\}} % set notation
\newcommand{\bigset}[2]{\big\{ #1 \;\big|\; #2 \big\}} % big set notation
\newcommand{\ssm}{\smallsetminus} % small set minus
\newcommand{\dotprod}[2]{\left\langle \, #1 \, \middle| \, #2 \, \right\rangle} % dot product
\newcommand{\bigdotprod}[2]{\big\langle \, #1 \, \big| \, #2 \, \big\rangle} % dot product
\newcommand{\symdif}{\,\triangle\,} % symmetric difference
\newcommand{\eqdef}{\mbox{\,\raisebox{0.2ex}{\scriptsize\ensuremath{\mathrm:}}\ensuremath{=}\,}} % :=
\newcommand{\defeq}{\mbox{~\ensuremath{=}\raisebox{0.2ex}{\scriptsize\ensuremath{\mathrm:}} }} % =:
\DeclareMathOperator{\conv}{conv} % convex hull
\newcommandx{\Perm}[2][1=\B_\circ]{\mathsf{Perm}^{#2}(#1)} % permutahedron
\newcommandx{\Zono}[2][1=\B_\circ]{\mathsf{Zono}^{#2}(#1)} % zonotope
\newcommandx{\Asso}[2][1=\B_\circ]{\mathsf{Asso}^{#2}(#1)} % associahedron
\newcommandx{\UnivAsso}[2][1=\B_\circ]{\mathsf{Asso}_{\mathrm{un}}^{#2}(#1)} % universal associahedron
\newcommand{\Para}[2][1=\B_\circ]{\mathsf{Para}^{#2}(#1)} % parallelepiped
\DeclareMathOperator{\face}{\mathbf{F}} % face of the permutahedron
\newcommand{\Fan}{\mathcal{F}} % fan
\newcommand{\polygon}{\Omega} % polygon
\newcommand{\triangulation}{\mathrm{T}} % triangulation
\newcommand{\dissection}{\mathrm{D}} % dissection
\newcommand{\quadrilateral}{\mathrm{Q}} % quadrilateral in a triangulation
\newcommand{\accordion}{\mathrm{A}} % accordion in a triangulation
\newcommand{\zigzag}{\mathrm{Z}} % zigzag
\newcommand{\sign}[3]{\varepsilon({#1} \!\in\! {#2}, {#3})} % sign
\newcommand{\bigsign}[3]{\varepsilon \big( {#1} \!\in\! {#2}, {#3} \big)} % sign
\newcommand{\SSS}{\reflectbox{$\mathsf{Z}$}} % S
\newcommand{\ZZZ}{\mathsf{Z}} % Z
\newcommand{\gvector}[2]{\mathbf{g}(#1, #2)} % g-vector of the cluster variable #2 with respect to the initial cluster #1
\newcommand{\biggvector}[2]{\mathbf{g} \big( #1, #2 \big)} % g-vector of the cluster variable #2 with respect to the initial cluster #1
\newcommand{\gvectors}[2]{\mathbf{g}(#1, #2)} % g-vectors of the cluster #2 with respect to the initial cluster #1
\newcommandx{\gvectorFan}[1][1=\B_\circ]{\mathcal{F}^\mathbf{g}(#1)} % g-vector fan
\newcommandx{\UnivgvectorFan}[1][1=\B_\circ]{\mathcal{F}_{\mathrm{un}}^\mathbf{g}(#1)} % g-vector fan

\newcommand{\cvector}[3]{\mathbf{c}(#1, #3 \in #2)} % c-vector of the cluster variable #3 in the cluster #2 with respect to the initial cluster #1
\newcommand{\bigcvector}[3]{\mathbf{c} \big( #1, #3 \in #2 \big)} % c-vector of the cluster variable #3 in the cluster #2 with respect to the initial cluster #1
\newcommand{\cvectors}[2]{\mathbf{c}(#1, #2)} % c-vectors of the cluster #2 with respect to the initial cluster #1
\newcommand{\allcvectors}[1]{\mathbf{C}(#1)} % all c-vectors with respect to the initial cluster #1
\newcommandx{\cvectorFan}[1][1=\B_\circ^\vee]{\mathcal{F}^\mathbf{c}(#1)} % c-vector fan
\newcommandx{\CoxeterFan}[1][1=\B_\circ^\vee]{\mathcal{F}^\mathbf{Cox}(#1)} % Coxeter fan

\newcommand{\uvector}[3]{\mathbf{u}(#1, #3 \in #2)} % u-vector of the cluster variable #3 in the cluster #2 with respect to the initial cluster #1
\newcommand{\biguvector}[3]{\mathbf{u} \big( #1, #3 \in #2 \big)} % u-vector of the cluster variable #3 in the cluster #2 with respect to the initial cluster #1

\newcommand{\rhs}[1]{k(#1)} % right hand side of zonotope

\newcommand{\point}[3]{\mathbf{p}^{#3}(#1, #2)} % vertex of the #1-associahedron corresponding to the cluster #2
\newcommand{\bigpoint}[3]{\mathbf{p}^{#3} \big( #1, #2 \big)} % vertex of the #1-associahedron corresponding to the cluster #2
\newcommand{\Univpoint}[3]{\mathbf{p}_{\mathrm{un}}^{#3}(#1, #2)} % vertex of the #1-associahedron corresponding to the cluster #2
\newcommand{\Univbigpoint}[3]{\mathbf{p}_{\mathrm{un}}^{#3} \big( #1, #2 \big)} % vertex of the #1-associahedron corresponding to the cluster #2
\newcommand{\HS}[3]{\mathbf{H}_{\le}^{#3}(#1, #2)} % half space
\newcommand{\bigHS}[3]{\mathbf{H}_{\le}^{#3} \big( #1, #2 \big)} % half space
\newcommand{\Hyp}[3]{\mathbf{H}_{=}^{#3}(#1, #2 )} % hyperplane
\newcommand{\bigHyp}[3]{\mathbf{H}_{=}^{#3} \big( #1, #2 \big)} % hyperplane
\newcommand{\coordSubspace}[1]{\mathbf{H}_{#1}} % coordinate hyperplane
\newcommand{\mi}{-} % bottom triangulation
\newcommand{\ma}{+} % top triangulation
\newcommand{\B}{\mathrm{B}} % b-matrix
\newcommand{\A}[1]{\mathrm{A}({#1})} % b-matrix
\newcommand{\D}{\mathrm{D}} % symmetrizer
\newcommand{\Trop}[1]{\mathbb{P}_{#1}} % tropical semifield
\newcommand{\positiveExponents}[1]{\left\{#1\right\}_+} % positive exponents in tropical semifield
\newcommand{\negativeExponents}[1]{\left\{#1\right\}_-} % negative exponents in tropical semifield
\newcommand{\seed}{\Sigma} % a cluster
\newcommand{\cluster}{\mathrm{X}} % a cluster
\newcommandx{\clusters}{\mathcal{X}} % all clusters
\newcommand{\coefficients}{\mathrm{P}} % a coefficients tuple
\newcommandx{\clusterAlgebra}[2][1=\B_\circ, 2=\coefficients_\circ]{\mathcal{A}\!\left(#1,#2\right)} % cluster algebra
\newcommandx{\clusterComplex}[1][1=\B_\circ]{\Delta(#1)} % cluster complex
\newcommandx{\principalClusterAlgebra}[1][1=\B_\circ]{\mathcal{A}_{\mathrm{pr}}\!\left(#1\right)} % principal coefficients cluster algebra
\newcommandx{\coefficientFreeClusterAlgebra}[1][1=\B_\circ]{\mathcal{A}_{\mathrm{fr}}\!\left(#1\right)} % coefficient free cluster algebra
\newcommandx{\universalClusterAlgebra}[1][1=\B_\circ]{\mathcal{A}_{\mathrm{un}}\!\left(#1\right)} % principal coefficients cluster algebra
\newcommandx{\clusterVariables}[1][1=\B_\circ]{\mathcal{X}(#1)} % all cluster variables
\newcommandx{\greenMutationsGraph}[1][1=\B_\circ]{\mathcal{G}(#1)} % directed graph of green mutations
\newcommand{\simpleRoot}{\alpha} % simple root
\newcommand{\fundamentalWeight}{\omega} % fundamental weights
\newcommand{\F}{h} % the magic F-function that we knew since 10 years
\newcommand{\signDep}[3]{\varepsilon(#1,#2,#3)} % the choice in the dependence
\newcommand{\bigsignDep}[3]{\varepsilon \big( #1,#2,#3 \big)} % the choice in the dependence
\newcommand{\compatibilityDegree}[2]{(\,#1\,\|\,#2\,)} % compatibility degree
\newcommand{\wo}{w_\circ} % longest element

\newcommand{\fref}[1]{Figure~\ref{#1}} % reference figures
\newcommand{\ie}{\textit{i.e.,}~} % id est
\newcommand{\eg}{\textit{e.g.,}~} % exempli gratia
\newcommand{\apriori}{\textit{a priori}} % a priori
\definecolor{darkblue}{rgb}{0,0,0.7} % darkblue color
\newcommand{\darkblue}{\color{darkblue}} % darkblue command
\newcommand{\defn}[1]{\textsl{\darkblue #1}} % emphasis of a definition
\newcommand{\para}[1]{\medskip\noindent\textbf{#1}} % paragraph
\usepackage{todonotes}

\begin{document}

\begin{abstract}
This paper shows the polytopality of any finite type $\b{g}$-vector fan, acyclic or not.
In fact, for any finite Dynkin type~$\Gamma$, we construct a universal associahedron~$\UnivAsso[\Gamma]{}$ with the property that any $\b{g}$-vector fan of type~$\Gamma$ is the normal fan of a suitable projection of~$\UnivAsso[\Gamma]{}$.
\end{abstract}

\maketitle

%%%%%%%%%%%%%%%%%%%%%%%%%%%%%%%%%%%%%%

\section{Introduction}
\label{sec:introduction}

A \defn{generalized associahedron} is a polytope which realizes the cluster complex of a finite type cluster algebra of S.~Fomin and A.~Zelevinsky~\cite{FominZelevinsky-YSystems, FominZelevinsky-ClusterAlgebrasI, FominZelevinsky-ClusterAlgebrasII}.
Generalized associahedra were first constructed by F.~Chapoton, S.~Fomin and A.~Zelevinsky~\cite{ChapotonFominZelevinsky} using the $\b{d}$-vector fans of~\cite{FominZelevinsky-YSystems}.
Further realizations were obtained by C.~Hohlweg, C.~Lange and H.~Thomas~\cite{HohlwegLangeThomas} based on the Cambrian lattices  of N.~Reading~\cite{Reading-CambrianLattices} and Cambrian fans of N.~Reading and D.~Speyer~\cite{ReadingSpeyer}.
These constructions were later revisited by S.~Stella~\cite{Stella} using an approach similar to the original one of~\cite{ChapotonFominZelevinsky}, and by V.~Pilaud and C.~Stump~\cite{PilaudStump-brickPolytope} via brick polytopes.

The realizations of~\cite{HohlwegLangeThomas} start from an acyclic initial exchange matrix~$\B_\circ$, and construct a generalized associahedron~$\Asso{}$ whose normal fan is the $\b{g}$-vector fan of the cluster algebra~$\principalClusterAlgebra$ with principal coefficients at~$\B_\circ$.
They rely on a combinatorial understanding of the $\b{g}$-vector fans as Cambrian fans~\cite{ReadingSpeyer}.
A major obstruction in dropping the acyclicity assumption in this approach is that this combinatorial description is only partially available beyond acyclic cases~\cite{ReadingSpeyer-cyclic}.
Therefore, it remained a challenging open problem, since the appearance of generalized associahedra, to construct polytopal realizations of all finite type $\b{g}$-vector fans including the cyclic cases.
This paper answers this problem.

\begin{theorem}
For any finite type initial exchange matrix~$\B_\circ$, the $\b{g}$-vector fan~$\gvectorFan$ with respect to~$\B_\circ$ is the normal fan of a generalized associahedron~$\Asso{}$.
\end{theorem}

When we start from an acyclic initial exchange matrix, our construction precisely recovers the associahedra of~\cite{HohlwegLangeThomas, Stella, PilaudStump-brickPolytope}.
These can all be obtained by deleting inequalities from the facet description of the permutahedron of the corresponding finite reflection group.
The main difficulty to extend the previous approach to arbitrary initial exchange matrices lies in the fact that this property, intriguing as it might be, is essentially a coincidence.
First, the hyperplane arrangement~$\cH$ supporting the $\b{g}$-vector fan is no longer the Coxeter arrangement of a finite reflection group.
Even worse, we prove that the generalized associahedron~$\Asso{}$ usually cannot be obtained by deleting inequalities in the facet description of any zonotope whose normal fan is~$\cH$.
This behavior already appears in type~$D_5$.

To overcome this situation, we develop an alternative approach based on a uniform understanding of the linear dependences among adjacent cones in the $\b{g}$-vector fan.
In fact, not only we cover uniformly all finite type initial exchange matrices, but we actually can treat them simultaneously with a universal object.

\begin{theorem}
\label{thm:universalAssociahedronIntro}
For any given finite Dynkin type~$\Gamma$, there exists a \defn{universal associahedron}~$\UnivAsso[\Gamma]{}$ such that, for any initial exchange matrix~$\B_\circ$ of type~$\Gamma$, the generalized associahedron~$\Asso{}$ is a suitable projection of the universal associahedron~$\UnivAsso[\Gamma]{}$.
In particular, all $\b{g}$-vector fans of type~$\Gamma$ are sections of the normal fan of the universal associahedron~$\UnivAsso[\Gamma]{}$.
\end{theorem}

This universal associahedron provides a tool to study simultaneously geometric properties of all generalized associahedra of a given finite Dynkin type.
For example, it is known that the vertex barycenters of all generalized associahedra of~\cite{HohlwegLangeThomas} lie at the origin.
In type~$A$, this property was observed by F.~Chapoton for J.-L.~Loday's realization of the classical associahedron~\cite{Loday}, conjectured for all associahedra of C.~Hohlweg and C.~Lange in~\cite{HohlwegLange}, proved by C.~Hohlweg, J.~Lortie and A.~Raymond~\cite{HohlwegLortieRaymond} and revisited by C.~Lange and V.~Pilaud in~\cite{LangePilaud}.
For arbitrary acyclic finite types, it was conjectured by C.~Hohlweg, C.~Lange and H.~Thomas in~\cite{HohlwegLangeThomas} and proved by V.~Pilaud and C.~Stump using the brick polytope approach~\cite{PilaudStump-barycenter}.
In the present paper, we use the universal associahedron to extend this surprising property to all generalized associahedra~$\Asso{}$. % in Theorem~\ref{thm:barycenter}.

\begin{theorem}
The origin is the vertex barycenter of the universal associahedron~$\UnivAsso[\Gamma]{}$, and thus of all generalized associahedra~$\Asso{}$.
\end{theorem}

The paper is organized as follows.
In Section~\ref{sec:prerequisitesCA}, we collect all the definitions and properties of finite type cluster algebras needed in this paper.
In Section~\ref{sec:prerequisitesFans}, we recall convenient criteria to check that a collection of cones forms a polyhedral fan and that a simplicial fan is the normal fan of a polytope.
Based on a precise understanding of the linear dependences in $\b{g}$-vectors of adjacent cones described in Section~\ref{sec:gvectorFan}, we prove the polytopality of all finite type $\b{g}$-vector fans in Section~\ref{sec:polytopality}.
Section~\ref{sec:specificFamilies} is devoted to two special cases: that of acyclic initial exchange matrices for which our construction yields the same generalized associahedra as~\cite{HohlwegLangeThomas, Stella, PilaudStump-brickPolytope}, and that of type~$A$ which presents several remarkable features.
In particular we prove that the facet description of the associahedron~$\Asso{}$ is contained in the facet description of the corresponding zonotope~$\Zono{}$ for any initial exchange matrix~$\B_\circ$ of type~$A$.
Further properties of our generalized associahedra are explored in Section~\ref{sec:properties}, including their connection to green mutations (Section~\ref{subsec:lattice}), the construction of the universal associahedron (Section~\ref{subsec:universalAssociahedron}), its vertex barycenter (Section~\ref{subsec:barycenter}), and a discussion on the relation between~$\Asso{}$ and~$\Zono{}$ (Section~\ref{subsec:zonotope}).

%%%%%%%%%%%%%%%%%%%%%%%%%%%%%%%%%%%%%%

\section{Finite type cluster algebras}
\label{sec:prerequisitesCA}

%%%%%%%%

We begin by recalling some standard notions on cluster algebras simplifying, whenever possible, our notations to deal with the case at hand.
This section can be used as a compendium of the results concerning finite type that are scattered through the literature. We refer to~\cite{FominZelevinsky-ClusterAlgebrasIV} for a general treatment of cluster algebras.

\subsection{Cluster algebras}

We will be working in the ambient field~$\Q(x_1, \dots, x_n, p_1, \dots, p_m)$ of rational expressions in~$n+m$ variables with coefficients in~$\Q$ and we denote by~$\Trop{m}$ its abelian multiplicative subgroup generated by the elements~$\{p_i\}_{i \in [m]}$.
Given~$p = \prod_{i \in [m]} p_i^{a_i} \in \Trop{m}$ we will write
\[
\positiveExponents{p} \eqdef \prod_{i \in [m]} p_i^{\max(a_i,0)}
\qquad\text{and}\qquad
\negativeExponents{p} \eqdef \prod_{i \in [m]} p_i^{-\min(a_i,0)}
\]
so that~$p = \positiveExponents{p} \negativeExponents{p}^{-1}$.

A \defn{seed}~$\seed$ is a triple~$(\B, \coefficients, \cluster)$ consisting of an exchange matrix, a coefficient tuple, and a cluster:
\begin{itemize}
\item the \defn{exchange matrix}~$\B$ is an integer~$n \times n$ skew-symmetrizable matrix, \ie such that there exist a diagonal matrix~$\D$ with~$-\B\D = (\B\D)^T$,
\item the \defn{coefficient tuple}~$\coefficients$ is any subset of~$n$ elements of~$\Trop{m}$,
\item the \defn{cluster}~$\cluster$ is a set of \defn{cluster variables}, $n$~rational functions in the ambient field that are algebraically independent over~$\Q(p_1, \dots, p_m)$.
\end{itemize}
To shorten our notation we think of rows and columns of~$\B$, as well as elements of~$\coefficients$, as being labeled by the elements of~$\cluster$: we write~$\B = (b_{xy})_{x,y \in \cluster}$ and~$\coefficients = \{p_x\}_{x \in \cluster}$.
Moreover we say that a cluster variable~$x$ (resp.~a coefficient~$p$) belongs to~$\seed = (\B, \coefficients, \cluster)$ to mean~$x \in \cluster$ (resp.~$p \in \coefficients$).

Given a seed~$\seed = (\B, \coefficients, \cluster)$ and a cluster variable~$x \in \seed$, we can construct a new seed~$\mu_x(\seed) = \seed' = (\B', \coefficients', \cluster')$ by \defn{mutation} in direction~$x$, where:
\begin{itemize}
\item the new cluster $\cluster'$ is obtained from $\cluster$ by replacing $x$ with the cluster variable $x'$ defined by the following \defn{exchange relation}:
\[
x x' = \positiveExponents{p_x} \prod_{\substack{y \in \cluster \\ b_{xy}  > 0}} y^{b_{xy}} + \negativeExponents{p_x} \prod_{\substack{y \in \cluster \\ b_{xy}  <0}} y^{-b_{xy}}
\]
and leaving the remaining cluster variables unchanged so that $\cluster \ssm \{x\} = \cluster' \ssm \{x'\}$.

\item the row (resp.~column) of~$\B'$ indexed by~$x'$ is the negative of the row (resp.~column) of~$\B$ indexed by~$x$, while all the other entries satisfy:
\[
b'_{yz} = b_{yz} + \frac{1}{2}\big(|b_{yx}| b_{xz} + b_{yx}|b_{xz}|\big),
\]

\item the elements of the new coefficient tuple $\coefficients'$ are 
\[
p'_y =
\begin{cases}
	p_x^{-1}  & \text{if } y = x', \\
	p_y\negativeExponents{p_x}^{b_{xy}}  & \text{if } y \neq x' \text{ and } b_{xy} \leq 0, \\
	p_y\positiveExponents{p_x}^{b_{xy}}  & \text{if } y \neq x' \text{ and } b_{xy} > 0,
\end{cases}
\]
\end{itemize}
A straightforward computation shows that mutations are involutions, \ie $\mu_{x'}(\mu_x(\seed)) = \seed$ so they define an equivalence relation on the collection of all seeds.

Fix a seed $\seed_\circ = (\B_\circ, \coefficients_\circ, \cluster_\circ)$ and call it \defn{initial}.
Up to  an automorphism of the ambient field we will assume that $\cluster_\circ = \{x_1,\dots,x_n\}$ and drop $\cluster_\circ$ from our notation.

\begin{definition}[{\cite[Def.~2.11]{FominZelevinsky-ClusterAlgebrasIV}}]
The (geometric type) \defn{cluster algebra} $\clusterAlgebra$ is the \mbox{$\Z\Trop{m}$-sub}\-ring of the ambient field generated by all the cluster variables in all the seeds mutationally equivalent to the initial seed $\seed_\circ$.
\end{definition}

\begin{example}
The simplest possible choice of coefficient tuple in the initial seed, namely $m=0$ and  $\coefficients_\circ=\{1\}_{i\in[n]}$, gives rise to the \defn{ cluster algebra without coefficients} which we will denote by~$\coefficientFreeClusterAlgebra$.
Note that this algebra, up to an automorphism of the ambient field $\Q(x_1, \dots, x_n)$, depends only on the mutation class of $\B_\circ$ and not on the exchange matrix itself.
The appearance of $\B_\circ$ in the notation $\coefficientFreeClusterAlgebra$ is just to fix the embedding inside the ambient field.
\end{example}

\subsection{Finite type}

We will only be dealing with cluster algebras of \defn{finite type} \ie cluster algebras having only a finite number of cluster variables.
As it turns out, being of finite type is a property that depends only on the exchange matrix in the initial seed and not on the coefficient tuple.

The \defn{Cartan companion} of an exchange matrix $\B$ is the symmetrizable matrix $\A{\B}$ given by:
\[
a_{xy} = 
\begin{cases}
	2 & \text{if } x = y, \\
	-|b_{xy}| & \text{otherwise.}
\end{cases}
\]

\begin{theorem}[{\cite[Thm.~1.4]{FominZelevinsky-ClusterAlgebrasII}}]
The cluster algebra $\clusterAlgebra$ is of finite type if and only if there exists an exchange matrix $\B$ obtained by a sequence of mutations from $\B_\circ$ such that its Cartan companion is a Cartan matrix of finite type.
Moreover the type of $\A{\B}$ is uniquely determined by $B_\circ$: if $\B'$ is any other exchange matrix obtained by mutation from $\B_\circ$ and such that $\A{\B'}$ is a finite type Cartan matrix then $\A{\B'}$ and $\A{\B}$ are related by a simultaneous permutation of rows and columns.
\end{theorem}

In accordance with the above statement, when talking about the \defn{(cluster) type} of $\clusterAlgebra$ or $\B_\circ$ we will refer to the Cartan type of $\A{\B}$.
We reiterate that the Cartan type of~$\A{\B_\circ}$ need not be finite: being of finite type is a property of the mutation class.

For a finite type cluster algebra~$\clusterAlgebra$, we will consider the root system of~$\A{\B_\circ}$.
To avoid any confusion later on let us state clearly the conventions we use in this paper: for us simple roots~$\{\simpleRoot_x\}_{x \in \cluster_\circ}$ and fundamental weights~$\{\fundamentalWeight_x\}_{x \in \cluster_\circ}$ are two basis of the same vector space~$V$; the matrix relating them is the Cartan matrix $\A{\B_\circ}$.
Fundamental weights are the dual basis to simple coroots~$\{\simpleRoot^\vee_x\}_{x \in \cluster_\circ}$, while simple roots are the dual basis to fundamental coweights~$\{\fundamentalWeight^\vee_x\}_{x \in \cluster_\circ}$; coroots and coweights are two basis of the dual space~$V^\vee$ and they are related by the transpose of the Cartan matrix.
This set of conventions is the standard one in Lie theory but it is not the one generally used in the setting of Coxeter groups~\cite[Chap.~4]{BjornerBrenti}.

A finite type exchange matrix~$\B_\circ$ is said to be \defn{acyclic} if~$\A{\B_\circ}$ is itself a Cartan matrix of finite type and \defn{cyclic} otherwise.
An acyclic finite type exchange matrix is said to be \defn{bipartite} if each of its rows (or equivalently columns) consists either of non-positive or non-negative entries.

\subsection{Principal coefficients, $\b{g}$- and $\b{c}$-vectors}

Among all the cluster algebras having a fixed initial exchange matrix, a central role is played by those with principal coefficients.
Indeed, thanks to the results in \cite{FominZelevinsky-ClusterAlgebrasIV}, they encode enough informations to understand all the other possible choices of coefficients.

\begin{definition}[{\cite[Def.~3.1]{FominZelevinsky-ClusterAlgebrasIV}}]
A cluster algebra is said to have \defn{principal coefficients} (at the initial seed) if its ambient field is $\Q(x_1, \dots, x_n, p_1, \dots, p_n)$ and the initial coefficient tuple consists of the generators of $\Trop{n}$ \ie $\coefficients_\circ = \{p_i\}_{i \in [n]}$.
In this case we will write $\principalClusterAlgebra$ for $\clusterAlgebra[B_\circ][\{p_i\}_{i \in [n]}]$, and we reindex the generators~$\{p_i\}_{i \in [n]}$ of $\Trop{n}$ by~$\{p_x\}_{x \in \cluster_\circ}$.
\end{definition}

In the above definition, it is important to specify that principal coefficients are with respect to a specific exchange matrix, even though it is usually omitted.
In other words $\principalClusterAlgebra$ and $\principalClusterAlgebra[\B_\circ']$ are in general different cluster algebras even when $\B_\circ$ and $\B_\circ'$ are related by mutations.

A notable property of cluster algebras with principal coefficients is that they are $\Z^n$-graded (in the basis~$\{\omega_x\}_{x \in \cluster_\circ}$ of~$V$).
The degree function~$\deg(\B_\circ,\cdot)$ on $\principalClusterAlgebra$ is obtained by setting
\[
\deg(\B_\circ,  x) \eqdef \fundamentalWeight_x
\qquad\text{and}\qquad
\deg(\B_\circ, p_x) \eqdef \sum_{y \in \cluster_\circ} -b_{yx} \fundamentalWeight_y
\]
for any $x\in\cluster_\circ$.
This assignment makes all exchange relations and all cluster variables in~$\principalClusterAlgebra$ homogeneous~\cite{FominZelevinsky-ClusterAlgebrasIV} and it justifies the definition of the following family of integer vectors associated to cluster variables.

\begin{definition}[\cite{FominZelevinsky-ClusterAlgebrasIV}]
The \defn{$\b{g}$-vector} of a cluster variable~$x \in \principalClusterAlgebra$ is its degree
\[
\gvector{\B_\circ}{x} \eqdef \deg(\B_\circ,x) \; \in V.
\]
We denote by $\gvectors{\B_\circ}{\seed} \eqdef \set{\gvector{\B_\circ}{x}}{x \in \seed}$ the set of $\b{g}$-vectors of the cluster variable in the seed~$\seed$ of~$\principalClusterAlgebra$.
\end{definition}

The next definition gives another family of integer vectors, introduced implicitly in \cite{FominZelevinsky-ClusterAlgebrasIV}, that are relevant in the structure of~$\principalClusterAlgebra$.

\begin{definition}
Given a seed~$\seed$ in~$\principalClusterAlgebra$, the \defn{$\b{c}$-vector} of a cluster variable $x \in \seed$ is the vector
\[
\cvector{\B_\circ}{\seed}{x} \eqdef \sum_{y \in \cluster_\circ} c_{yx} \, \simpleRoot_y \; \in V
\]
of exponents of~$p_x = \prod_{y \in \cluster_\circ} (p_y)^{c_{yx}}$.
Let~$\cvectors{\B_\circ}{\seed} \eqdef \set{\cvector{\B_\circ}{\seed}{x}}{x \in \seed}$ denote the set of $\b{c}$-vectors of a seed~$\seed$.
Finally, let~$\allcvectors{\B_\circ} \eqdef \bigcup_{\seed}{\cvectors{\B_\circ}{\seed}}$ denote the set of all~$\b{c}$-vectors in~$\principalClusterAlgebra$.
\end{definition}

It is worth spending few words here to emphasize the fact that, contrary to what happens for $\b{g}$-vectors, $\b{c}$-vectors are not attached to cluster variables {\it per se} but depends on the seed in which the given cluster variable sits.

An important features of $\b{c}$-vectors is that their entries weekly agree in sign.
This is one of the various reformulation of the \defn{sign-coherence} conjecture of~\cite{FominZelevinsky-ClusterAlgebrasIV} recently established in full generality by~\cite{GrossHackingKeelKontsevich}.
In the setting of finite type cluster algebras, this result can also be deduced in several ways from earlier works: one proof is to combine \cite{DerksenWeymanZelevinsky} with~\cite{Demonet}; another one is to use surfaces~\cite{FominShapiroThurston, FominThurston} and orbifolds~\cite{FeliksonShapiroTumarkin} to study types~$A_n, B_n, C_n$ and~$D_n$, and to deal with exceptional types by direct inspection.
Here we prefer to observe it as a corollary of the following theorem that will be useful later on to justify our notation.

\begin{theorem}[{\cite[Thm.~1.3]{NakanishiStella}}]
\label{thm:roots}
The $\b{c}$-vectors of the finite type cluster algebra $\principalClusterAlgebra$ are roots in the root system whose Cartan matrix is $\A{\B_\circ}$.
\end{theorem}

Again note that, since~$\A{\B_\circ}$ may be not of finite type, the root system in this statement is, in general, not finite.
For example, for the cyclic type~$A_3$ exchange matrix, the Cartan companion~$\A{\B_\circ}$ is of affine type~$A_2^{(1)}$, see its Coxeter arrangement in \fref{fig:exmFans}\,(top right).
More precisely, it is finite if and only if~$\B_\circ$ is acyclic.
We will discuss in more details the relation of $\b{c}$-vectors with root systems in Remark~\ref{rem:shi}.

Another consequence of \cite[Thm.~1.3]{NakanishiStella} is that in a cluster algebra of finite type, there are $nh$ distinct $\b{c}$-vectors where~$h$ is the Coxeter number of the given finite type.
We remind to our reader that the same algebra has $(h+2)n/2$ distincts $\b{g}$-vectors, one for each cluster variable.

\medskip
Our next task in this section is to discuss a duality relation in between $\b{c}$-vectors and $\b{g}$-vectors.
A first step is to recall the notion of the \defn{cluster complex} of $\clusterAlgebra$: it is the abstract simplicial complex whose vertices are the cluster variables of $\clusterAlgebra$ and whose facets are its clusters.
As it turns out, at least in the finite type cases, this complex is independent of the choice of coefficients, see~\cite[Thm.~1.13]{FominZelevinsky-ClusterAlgebrasII} and \cite[Conj.~4.3]{FominZelevinsky-ClusterAlgebrasIV}.
In particular this means that, up to isomorphism, there is only one cluster complex for each finite type: the one associated to $\coefficientFreeClusterAlgebra$.
We will use this remark later on to relate cluster variables of different cluster algebras of the same finite type.
Note also that, again when $\clusterAlgebra$ is of finite type, the cluster complex is a \defn{pseudomanifold}~\cite{FominZelevinsky-ClusterAlgebrasII}.

For a skew-symmetrizable exchange matrix~$\B_\circ$, the matrix~$\B_\circ^\vee \eqdef -\B_\circ^T$ is still skew-symmetrizable.
The cluster algebras~$\principalClusterAlgebra$ and~$\principalClusterAlgebra[\B_\circ^\vee]$ can be thought as \defn{dual} to each other.
Indeed their types are Langlands dual of each other.
Moreover their cluster complexes are isomorphic: by performing the same sequence of mutations we can identify any cluster variable~$x$ of~$\principalClusterAlgebra$ with a cluster variable~$x^\vee$ of~$\principalClusterAlgebra[\B_\circ^\vee]$, and any seed $\seed$ in $\principalClusterAlgebra$ with a seed $\seed^\vee$ in~$\principalClusterAlgebra[\B_\circ^\vee]$.
More importantly the following crucial property holds.

\begin{theorem}[{\cite[Thm.~1.2]{NakanishiZelevinsky}}]
\label{thm:duality}
For any seed~$\seed$ of~$\principalClusterAlgebra$, let $\seed^\vee$ be its dual in $\principalClusterAlgebra[\B_\circ^\vee]$.
Then the $\b{g}$-vectors~$\gvector{\B_\circ}{\seed}$ of the cluster variables in~$\seed$ and the $\b{c}$-vectors~$\cvectors{\B_\circ^\vee}{\seed^\vee}$ of the cluster variables in~$\seed^\vee$ are dual bases, \ie
\[
\bigdotprod{\gvector{\B_\circ}{x}}{\cvector{\B_\circ^\vee}{\seed^\vee}{y^\vee}} = \delta_{x=y}
\]
for any two cluster variables~$x,y \in \seed$.
\end{theorem}

In view of the above results, and since $\A{\B_\circ^\vee}=\A{\B_\circ}^T$, the $\b{c}$-vectors of a finite type cluster algebra $\principalClusterAlgebra[\B_\circ^\vee]$ can be understood as coroots for $\A{\B_\circ}$ so that the $\b{g}$-vectors of $\principalClusterAlgebra$ become weights.
This justify our choice to define $\b{g}$-vectors in the weight basis.

\subsection{Coefficient specialization and universal cluster algebra}

We now want to relate, within a given finite type, cluster algebras with different choices of coefficients.
Pick a finite type exchange matrix $\B_\circ$ and let $\clusterAlgebra\subset\Q(x_1,\dots,x_n,p_1,\dots,p_m)$ and $\clusterAlgebra[\B_\circ][\overline{\coefficients_\circ}]\subset\Q(\overline{x_1},\dots,\overline{x_n},\overline{p_1},\dots,\overline{p_\ell})$ be any two cluster algebras having~$\B_\circ$ as exchange matrix in their initial seed.
As we said, cluster variables and seeds in these two algebras are in bijection because their cluster complexes are isomorphic.
Let us write
\[
x \longleftrightarrow \overline{x}
\quad\quad \text{and}\quad\quad
\seed \longleftrightarrow \overline{\seed}
\]
for this bijection.
We will say that~$\clusterAlgebra[\B_\circ][\overline{\coefficients_\circ}]$ is obtained from~$\clusterAlgebra$ by a \defn{coefficient specialization} if there exist a map of abelian groups~$\eta:\Trop{m}\rightarrow\Trop{\ell}$ such that, for any~$p_x$ in some seed~$\seed$ of~$\clusterAlgebra$ 
\[
\eta(\positiveExponents{p_x}) = \positiveExponents{\overline{p_{\overline{x}}}}
\quad\quad \text{and} \quad\quad
\eta(\negativeExponents{p_x}) = \negativeExponents{\overline{p_{\overline{x}}}}
\]
and which extends in a unique way to a map of algebras that satisfy
\[
\eta(x) = \overline{x}.
\]
Note that this is not the most general definition (see~\cite[Def.~12.1 and Prop.~12.2]{FominZelevinsky-ClusterAlgebrasIV}) but it will suffice here.
Armed with the notion of coefficient specialization we can now introduce the last kind of cluster algebra of finite type we will need.

\begin{definition}[{\cite[Def.~12.3 and Thm.~12.4]{FominZelevinsky-ClusterAlgebrasIV}}]
Pick a finite type exchange matrix $\B_\circ$.
The \defn{cluster algebra with universal coefficients} $\universalClusterAlgebra$ is the unique (up to canonical isomorphism) cluster algebra such that any other cluster algebra of the same type as $\B_\circ$ can be obtained from it by a unique coefficient specialization.
\end{definition}

Let us insist on the fact that, in view of the universal property it satisfies, $\universalClusterAlgebra$ depends only on the type of~$\B_\circ$ and not on the exchange matrix~$\B_\circ$ itself.
We again keep $\B_\circ$ in the notation only to fix an embedding into the ambient field.

Rather than proving the existence and explaining the details of how such a universal algebra is built, we will recall here one of its remarkable properties that follows directly from the $\b{g}$-vector recursion~\cite[Prop.~4.2\,(v)]{NakanishiZelevinsky} and that we will need in our proofs later on.
Denote by $\clusterVariables$ the set of all cluster variables in $\universalClusterAlgebra$ and let $\{p[x]\}_{x\in\clusterVariables}$ be the generators\footnote{A note to the reader that might be scared of a circular reasoning here: the set of generators $\{p[x]\}_{x\in\clusterVariables}$ is just a collection of symbols and its cardinality can be precomputed: it depends only on the type of the algebra and not on its coefficients.} of $\Trop{|\clusterVariables|}$.

\begin{theorem}[{\cite[Theo.~10.12]{Reading-UniversalCoefficients}}]
The cluster algebra $\universalClusterAlgebra$ can be realized over $\Trop{|\clusterVariables|}$.
The coefficient tuple $\coefficients=\{p_x\}_{x \in \cluster}$ at each seed $\seed=(\B,\coefficients,\cluster)$ of $\universalClusterAlgebra$ is given by the formula
\[
p_x = \prod_{y \in \clusterVariables} \big(p[y]\big)^{[\gvector{\B^T}{y^T};{x^T}]}
\]
where we denote by $[\b{v};x]$ the $x$-th coefficient of a vector $\b{v}$ in the weight basis~$(\omega_x)_{x \in \cluster}$.
The bijection of the elements of $\clusterVariables$ with the cluster variables of $\principalClusterAlgebra[\B^T]$, appearing in the formula, is given by an isomorphism of the corresponding cluster complexes similar to the one discussed above.
\end{theorem}

\begin{remark}
\label{rem:universalToPrincipal}
In view of this result, it is straightforward to produce the coefficient specialization morphism to get any cluster algebra with principal coefficients of type~$\B_\circ$ from $\universalClusterAlgebra$.
Namely, for any seed~$\seed_\star = (\B_\star, \coefficients_\star, \cluster_\star)$ of $\universalClusterAlgebra$, we obtain~$\principalClusterAlgebra[\B_\star]$ by evaluating to $1$ all the coefficients~$p[y]$ corresponding to cluster variables~$y$ not in $\seed_\star$.
\end{remark}

We conclude this review giving an example: the cluster algebra of type $B_2$ with universal coefficients.
We will do so in terms of ``tall'' rectangular matrices to help readers, not familiar with the language we adopt here, recognize an hopefully more familiar setting.
Indeed, to pass from our seeds to the one consisting of extended exchange matrices and clusters with frozen variables, it suffices to observe that each $p\in\Trop{m}$ can be encoded in a vector.
In this way any $n$-tuple of elements of $\Trop{m}$ corresponds to a $m\times n$ integer matrix and one gets an extended exchange matrix by gluing it below the exchange matrix of the seed.
The frozen variables are the generators of $\Trop{m}$ and the rules of mutations we discussed become then the usual mutations of extended exchange matrices.
We prefer to use the notation we set up here following \cite{FominZelevinsky-ClusterAlgebrasII,FominZelevinsky-ClusterAlgebrasIV} because it makes more evident the distinction in between coefficients and cluster variables, and because it is more suited to deal with coefficient specializations.

\begin{example}
\label{exm:B2}
Consider the exchange matrix 
\[
\B_\circ = \begin{bmatrix} 0 & 1  \\ -2 & 0  \end{bmatrix}.
\]
Any cluster algebra built from this matrix will contain $6$ cluster variables.
We will call them $\clusterVariables = \{ x_1,x_2,x_3,x_4,x_5,x_6\}$.
The corresponding cluster algebra with universal coefficients will then be a subring of 
\[
\Q(x_1,x_2,p[x_1],p[x_2],p[x_3],p[x_4],p[x_5],p[x_6]).
\]
Namely it will be the cluster algebra
\[
\universalClusterAlgebra=\clusterAlgebra[\B_\circ][{\left\{\frac{p[x_1]p[x_5]p[x_6]^2}{p[x_3]},\frac{p[x_2]}{p[x_4]p[x_5]p[x_6]}\right\}}]
\]
\fref{fig:B2_example} shows the exchange graph of this algebra listing all the seeds in terms of extended exchange matrices.

\begin{figure}
	\capstart
	\centerline{\includegraphics[width=.9\textwidth]{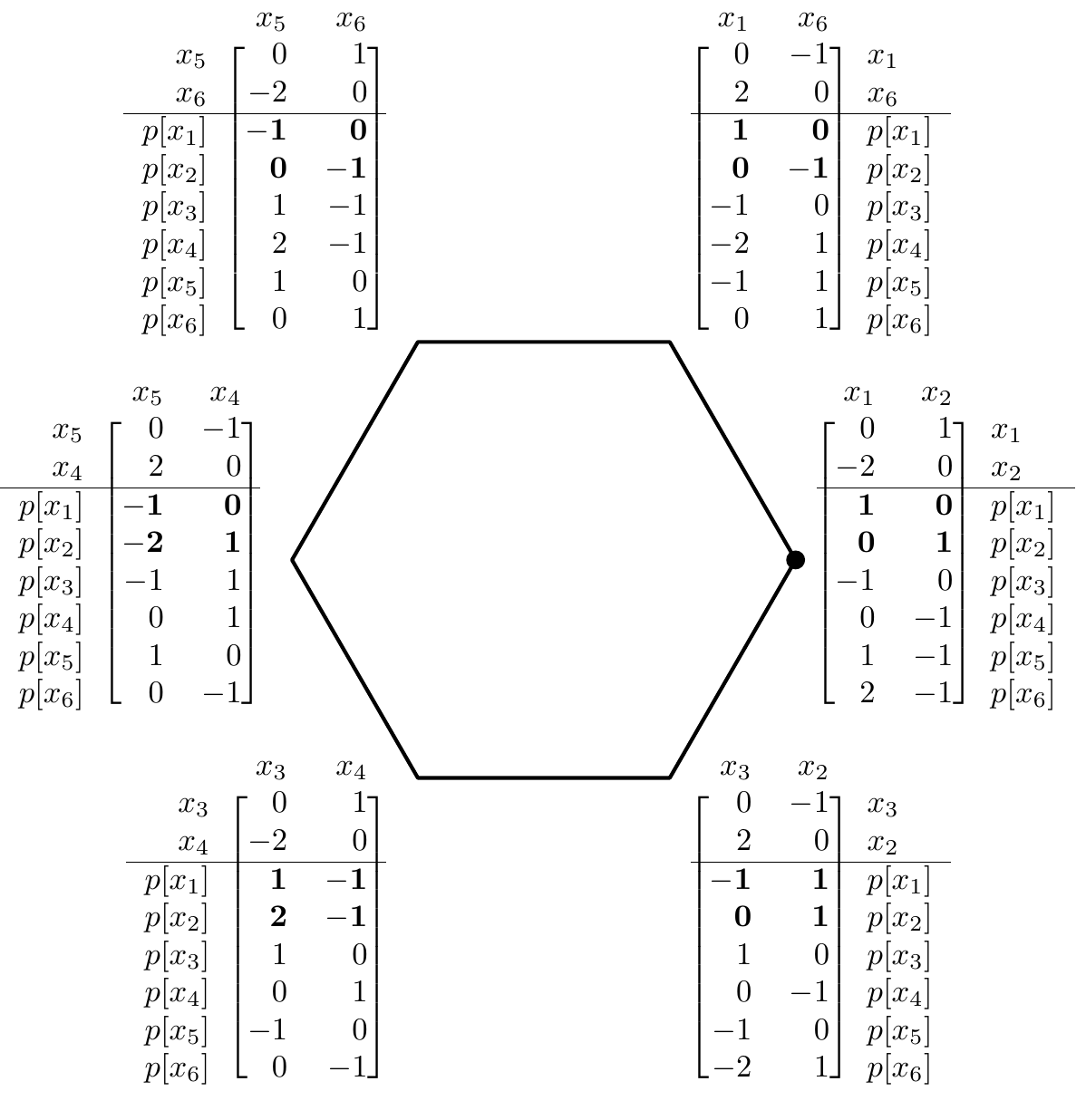}}
	\caption{The exchange graph of type $B_2$ with attached the rectangular matrices giving universal coefficients. In bold are highlighted the entries of the coefficient part that give the principal coefficient cluster algebra at the seed attached to the marked node.}
	\label{fig:B2_example}
\end{figure}
\end{example}

One final computational remark: there is a simple algorithm to compute one of the rectangular exchange matrices appearing in a cluster algebra of finite type with universal coefficients.
Let $\B$ be an exchange matrix of the given finite type having only non-negative entries above its diagonal; it is acyclic and, by \cite[Eqn.~(1.4)]{YangZelevinsky}, it corresponds to the \defn{Coxeter element} $c = s_1 \cdots s_n$ in the associated Weyl group (note that the labelling of simple roots may not be the standard one here).
Let~$\wo(c)$ denote the $c$-sorting word for~$\wo$, that is the lexicographically minimal reduced expression of~$\wo$ that appears as a subword of~$c^\infty$.
For each $s_i$ obtained reading from left to right the word $c\wo(c)$, repeat the following two steps:
\begin{itemize}
\item add a row to the bottom of $\B$ whose only non-zero entry is a $1$ in column $i$
\item replace $\B$ by its mutation in direction $i$.
\end{itemize}
The matrix obtained at the end of the procedure is the desired one.

%%%%%%%%%%%%%%%%%%%%%%%%%%%%%%%%%%%%%%

\section{Polyhedral geometry and fans}
\label{sec:prerequisitesFans}

The second ingredient of this paper is discrete geometry of polytopes and fans.
We refer to~\cite{Ziegler-polytopes} for a textbook on this topic.

\subsection{Polyhedral fans}

A \defn{polyhedral cone} is a subset of the vector space~$V$ defined equivalently as the positive span of finitely many vectors or as the intersection of finitely many closed linear halfspaces.
Throughout the paper, we write~$\R_{\ge0}\b{\Lambda}$ for the positive span of a set~$\b{\Lambda}$ of vectors of~$V$.
The \defn{faces} of a cone~$C$ are the intersections of~$C$ with its supporting hyperplanes.
In particular, the $1$-dimensional (resp.~codimension~$1$) faces of~$C$ are called~\defn{rays} (resp.~\defn{facets}) of~$C$.
A cone is \defn{simplicial} if it is generated by a set of independent vectors.

A \defn{polyhedral fan} is a collection~$\Fan$ of polyhedral cones of~$V$ such that
\begin{itemize}
\item if~$C \in \Fan$ and~$F$ is a face of~$C$, then~$F \in \Fan$,
\item the intersection of any two cones of~$\Fan$ is a face of both.
\end{itemize}
A fan is \defn{simplicial} if all its cones are, and \defn{complete} if the union of its cones covers the ambient space~$V$.
For a simplicial fan~$\Fan$ with rays~$\cX$, the collection~$\set{X \subseteq \cX}{\R_{\ge0}X \in \Fan}$  of generating sets of the cones of~$\Fan$ defines a \defn{pseudomanifold} (in other words, a pure and thin simplicial complex, \ie with a notion of flip).
The following statement characterizes which pseudomanifolds are complete simplicial fans.
A formal proof can be found \eg in~\cite[Coro.~4.5.20]{DeLoeraRambauSantos}.

\begin{proposition}
\label{prop:characterizationFan}
Consider a pseudomanifold~$\Delta$ with vertex set~$\cX$ and a set of vectors~$\big\{ \b{r}(x) \big\}_{x \in \cX}$ of~$V$.
For~$\cluster \in \Delta$, let~${\b{r}(\cluster) \eqdef \bigset{\b{r}(x)}{x \in \cluster}}$.
Then the collection of cones~${\bigset{\R_{\ge 0}\b{r}(\cluster)}{\cluster \in \Delta}}$ forms a complete simplicial fan if and~only~if
\begin{enumerate}
\item there exists a facet~$\cluster$ of~$\Delta$ such that~$\b{r}(\cluster)$ is a basis of~$V$ and such that the open cones~$\R_{> 0}\b{r}(\cluster)$ and~$\R_{> 0}\b{r}(\cluster')$ are disjoint for any facet~$\cluster'$ of~$\Delta$ distinct from~$\cluster$;
\item for any two adjacent facets~$\cluster, \cluster'$ of~$\Delta$ with~$\cluster \ssm \{x\} = \cluster' \ssm \{x'\}$, there is a linear dependence
\[
\gamma \, \b{r}(x) + \gamma' \, \b{r}(x') + \sum_{y \in \cluster \cap \cluster'} \delta_y \, \b{r}(y) = 0
\]
on~$\b{r}(\cluster \cup \cluster')$ where the coefficients~$\gamma$ and~$\gamma'$ have the same sign.
(When these conditions hold, these coefficients do not vanish and the linear dependence is unique up to rescaling.)
\end{enumerate}
\end{proposition}

\subsection{Polytopes and normal fans}

A \defn{polytope} is a subset~$P$ of~$V^\vee$ defined equivalently as the convex hull of finitely many points or as a bounded intersection of finitely many closed affine halfspaces.
The \defn{faces} of~$P$ are the intersections of~$P$ with its supporting hyperplanes.
In particular, the dimension~$0$ (resp.~dimension~$1$, resp.~codimension~$1$) faces of~$P$ are called \defn{vertices} (resp.~\defn{edges}, resp.~\defn{facets}) of~$P$.
The (outer) \defn{normal cone} of a face~$F$ of~$P$ is the cone in $V$ generated by the outer normal vectors of the facets of~$P$ containing~$F$.
The (outer) \defn{normal fan} of~$P$ is the collection of the (outer) normal cones of all its faces.
We say that a complete polyhedral fan in~$V$ is \defn{polytopal} when it is the normal fan of a polytope in~$V^\vee$.
The following statement provides a characterization of polytopality of complete simplicial fans.
It is a reformulation of regularity of triangulations of vector configurations, introduced in the theory of secondary polytopes~\cite{GelfandKapranovZelevinsky}, see also~\cite{DeLoeraRambauSantos}.
We present here a convenient formulation from~\cite[Lem.~2.1]{ChapotonFominZelevinsky}.

\begin{proposition}
\label{prop:characterizationPolytopalFan}
\label{prop:polytopalityFan}
Consider a pseudomanifold~$\Delta$ with vertex set~$\cX$ and a set of vectors~$\big\{ \b{r}(x) \big\}_{x \in \cX}$ of~$V$ such that~$\Fan \eqdef {\bigset{\R_{\ge 0}\b{r}(\cluster)}{\cluster \in \Delta}}$ forms a complete simplicial fan.
Then the following are equivalent:
\begin{enumerate}
\item $\Fan$ is the normal fan of a simple polytope in~$V^\vee$;
\item There exists a map~$\F: \cX \to \R_{> 0}$ such that for any two adjacent facets~$\cluster, \cluster'$ of~$\Delta$ with ${\cluster \ssm \{x\} = \cluster' \ssm \{x'\}}$, we have
\[
\gamma \, \F(x) + \gamma' \, \F(x') + \sum_{y \in \cluster \cap \cluster'} \delta_y \, \F(y) > 0,
\]
where
\[
\gamma \, \b{r}(x) + \gamma' \, \b{r}(x') + \sum_{y \in \cluster \cap \cluster'} \delta_y \, \b{r}(y) = 0
\]
is the unique (up to rescaling) linear dependence with~$\gamma, \gamma' > 0$ between the rays of~$\b{r}(\cluster \cup \cluster')$.
\end{enumerate}
Under these conditions, $\Fan$ is the normal fan of the polytope defined by
\[
\bigset{\b{v} \in V^\vee}{\dotprod{\b{r}(x)}{\b{v}} \le \F(x) \text{ for all } x \in \cX}.
\]
\end{proposition}

%%%%%%%%%%%%%%%%%%%%%%%%%%%%%%%%%%%%%%

\section{The $\b{g}$-vector fan}
\label{sec:gvectorFan}

We first recast a well known fact concerning the cones spanned by the $\b{g}$-vectors of a finite type cluster algebra with principal coefficients.
We insist on the fact that the following statement is valid for any finite type exchange matrix~$\B_\circ$, acyclic or not.

\begin{theorem}%[\cite{FominZelevinsky-YSystems, Reading-UniversalCoefficients}]
\label{thm:gvectorFan}
For any finite type exchange matrix~$\B_\circ$, the collection of cones 
\[
  \gvectorFan \eqdef \bigset{\R_{\ge0} \gvectors{\B_\circ}{\seed}}{\seed \text{ seed of } \principalClusterAlgebra},
\]
together with all their faces, forms a complete simplicial fan, called the \defn{$\b{g}$-vector fan} of~$\B_\circ$. % and denoted by~$\gvectorFan$.
\end{theorem}

There are several ways to find or deduce Theorem~\ref{thm:gvectorFan} from the literature.
First, it was established in the acyclic case in~\cite{ReadingSpeyer, YangZelevinsky, Stella} (see Example~\ref{exm:CambrianFan}).
As already observed by N.~Reading in~\cite[Thm.~10.6]{Reading-UniversalCoefficients}, the general case then follows from the initial seed recursion on $\b{g}$-vectors~\cite[Prop.~4.2\,(v)]{NakanishiZelevinsky}, valid thanks to sign-coherence.
A second proof would be to use the unique cluster expansion property of any vector in the weight lattice (following from the fact that cluster monomials are a basis of $\principalClusterAlgebra$ in finite type), and to use approximation by this lattice to show that any vector is covered exactly once by the interiors of the cones of the $\b{g}$-vector fan.
Note that contrarily to what sometimes appears in the literature, this approximation argument is subtle as it relies on the integrity of the $\b{g}$-vectors\footnote{For an illustration of the subtlety, consider the~$8$  cones in~$\R^3$ defined by the coordinate hyperplanes, rotate the $4$ cones with~$x \ge 0$ around the $x$-axis by~$\pi/4$, and finally rotate all the cones around the origin by an irrational rotation so that each rational direction belongs to the interior of one of the $8$ resulting cones. Then any vector in~$\Z^3$ belongs to a single cone, but the resulting cones do not form a fan (since the cones with~$x \ge 0$ intersect improperly those with~$x \le 0$).}.
Finally, another possible proof is to use Proposition~\ref{prop:characterizationFan}: the first point is a simplified version of the unique expansion property, and the second point is a consequence of the following description of the linear dependence between the $\b{g}$-vectors of two adjacent clusters, which will be crucial in the next section.

\begin{lemma}
\label{lem:linearDependence}
For any finite type~exchange matrix~$\B_\circ$ and any adjacent seeds~${(\B, \coefficients, \cluster)}$ and~${(\B', \coefficients', \cluster')}$ in $\principalClusterAlgebra$ with~$\cluster \ssm \{x\} = \cluster' \ssm \{x'\}$, the $\b{g}$-vectors of~$\cluster \cup \cluster'$ satisfy precisely one of the following two linear dependences
\[
\biggvector{\B_\circ}{x} + \biggvector{\B_\circ}{x'} = \sum_{\substack{y \in \cluster \cap \cluster' \\ b_{xy} < 0}} -b_{xy} \, \biggvector{\B_\circ}{y}
\quad\text{or}\quad
\biggvector{\B_\circ}{x} + \biggvector{\B_\circ}{x'} = \sum_{\substack{y \in \cluster \cap \cluster' \\ b_{xy} > 0}} b_{xy} \, \biggvector{\B_\circ}{y}.
\]
\end{lemma}

\begin{proof}
This is a straightforward consequence of the definition of $\b{g}$-vectors together with sign coherence.
Indeed all exchange relations in $\principalClusterAlgebra$ are homogeneous and
\[
x x' = \positiveExponents{p_x} \prod_{\substack{y \in \cluster \\ b_{xy}  > 0}} y^{b_{xy}} + \negativeExponents{p_x} \prod_{\substack{y \in \cluster \\ b_{xy}  <0}} y^{-b_{xy}}
\]
means that 
\begin{align*}
\deg(\B_\circ, x) + \deg(\B_\circ, x')
& = \deg(\B_\circ, \positiveExponents{p_x}) + \sum_{\substack{y \in \cluster \cap \cluster' \\ b_{xy} > 0}} b_{xy} \, \deg(\B_\circ, y) \\
& = \deg(\B_\circ, \negativeExponents{p_x}) + \sum_{\substack{y \in \cluster \cap \cluster' \\ b_{xy} < 0}} -b_{xy} \, \deg(\B_\circ, y).
\end{align*}
Now, by sign-coherence, exactly one of $\positiveExponents{p_x}$ and $\negativeExponents{p_x}$ is $1$ so that its degree is $0$.
\end{proof}

\begin{remark}
Note that which of the two possible linear dependences is satisfied by the $\b{g}$-vectors of~$\cluster \cup \cluster'$ depends on the initial exchange matrix~$\B_\circ$.
In particular, the geometry of the $\b{g}$-vector fan~$\gvectorFan$ changes as $\B_\circ$ varies within a given mutation class.
\end{remark}

For any finite type exchange matrix $\B_\circ$, the $\b{g}$-vector fan $\gvectorFan$ can be seen as a coarsening of two other fans naturally associated to $\A{\B_\circ}$.
Denote by $\CoxeterFan$ the \defn{dual Coxeter fan} \ie the fan of regions of the hyperplane arrangement given by the root system of $\A{\B_\circ^\vee}$.
Similarly let $\cvectorFan$ be the \defn{dual $\b{c}$-vector fan} \ie the fan of regions of the arrangement of hyperplanes orthogonal to all the $\b{c}$-vector of $\principalClusterAlgebra[\B_\circ^\vee]$.
By Theorem~\ref{thm:duality}, $\gvectorFan$ coarsens $\cvectorFan$ which, in turn, coarsens $\CoxeterFan$ by Theorem~\ref{thm:roots}.
See Figure~\ref{fig:exmFans} for examples of these fans for different exchange matrices of type $A_3$.

\begin{remark}
\label{rem:shi}
By further inspecting \cite[Thm.~1.3]{NakanishiStella} we can say more about $\cvectorFan$.
Indeed its supporting hyperplane arrangement contains the hyperplanes associated to small roots in the root system of~$\A{\B_\circ^\vee}$.
Therefore, it turns out that the dual $\b{c}$-vector fan~$\cvectorFan[\B_\circ^\vee]$ intersected with the Tits cone contains the \defn{Shi arrangement} for the root system of $\A{\B_\circ^\vee}$ (see \cite[Sect.~3.6~\&~Def.~3.18]{HohlwegNadeauWilliams} for a review on the topic).
In order to see that, it is enough to compare the description of the possible supports of the $\b{c}$-vectors (in terms of simple roots) given in~\cite[Thm.~1.3]{NakanishiStella} with the description of the possible supports of small roots given in~\cite{Brink}.
\end{remark}

\begin{example}
\label{exm:CambrianFan}
When the exchange matrix~$\B_\circ$ is acyclic, the $\b{g}$-vector fan is the Cambrian fan constructed by N.~Reading and D.~Speyer~\cite{ReadingSpeyer}, while the dual $\b{c}$-vector fan is the type~$\A{\B_\circ^\vee}$ Coxeter fan.
Section~\ref{subsec:acyclic} provides a detailed discussion of the acyclic case.
\end{example}

\begin{example}
\label{exm:specific1}
\fref{fig:gvectorFans} illustrates the $\b{g}$-vector fans for the initial exchange matrices
\\[.3cm]
\centerline{
\begin{tabular}{c@{\qquad}c@{\qquad}c}
	$\begin{bmatrix} 0 & -1 & 1 \\ 1 & 0 & -1 \\ -1 & 1 & 0 \end{bmatrix}$ & and & $\begin{bmatrix} 0 & -1 & 2 \\ 1 & 0 & -2 \\ -1 & 1 & 0 \end{bmatrix}$ \\[.6cm]
	(type~$A_3$ cyclic) & & (type~$C_3$ cyclic)
\end{tabular}
} \\[.2cm]
Note that these matrices are the only examples of $3$-dimensional cyclic exchange matrices (up to duality and simultaneous permutations of rows and columns).

\begin{figure}
	\capstart
	\centerline{\includegraphics[width=.45\textwidth]{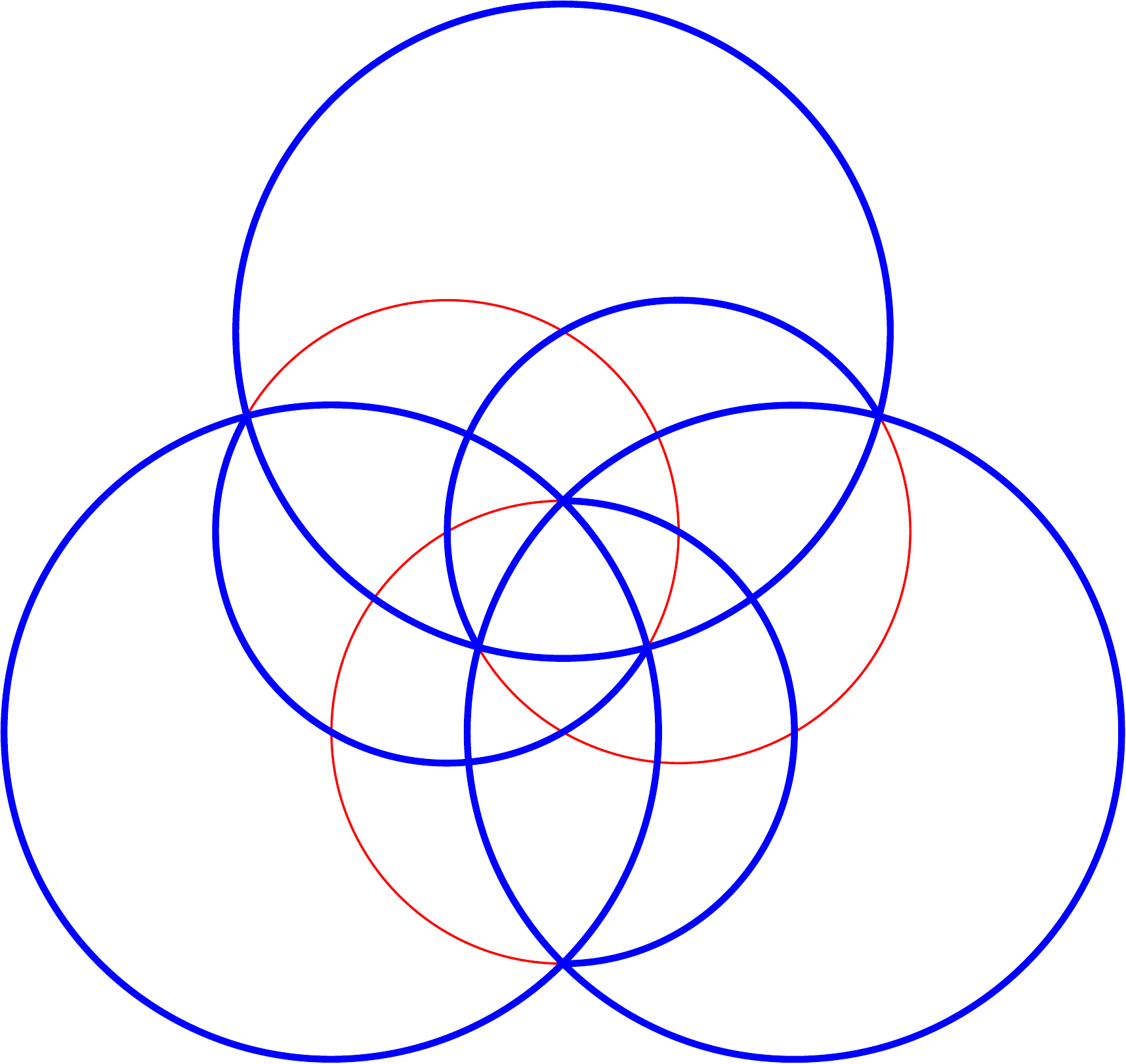} \qquad \includegraphics[width=.45\textwidth]{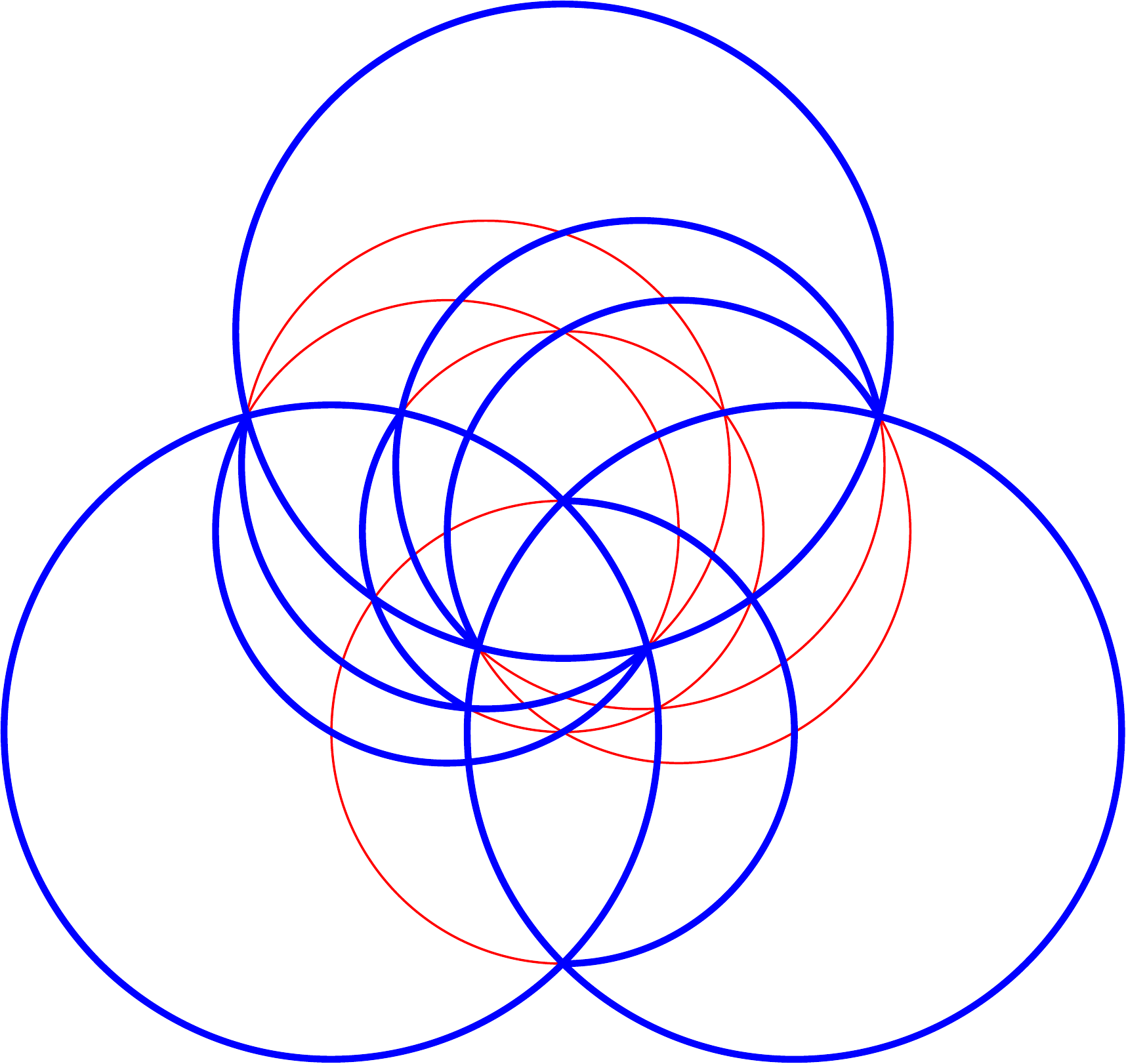}}
  \caption{The dual $\b{c}$-vector fan~$\cvectorFan[\B_\circ^\vee]$ (thin red) and $\b{g}$-vector fan~$\gvectorFan$ (bold blue) for the type~$A_3$~(left) and type~$C_3$~(right) cyclic initial exchange matrices. Each $3$-dimensional fan is intersected with the unit sphere and stereographically projected to the plane from the pole~$(-1,-1,-1)$. We use this drawing convention for all $3$-dimensional fans in this paper.}
	\label{fig:gvectorFans}
\end{figure}
\end{example}

%%%%%%%%%%%%%%%%%%%%%%%%%%%%%%%%%%%%%%

\section{Polytopality}
\label{sec:polytopality}

In this section, we show that the $\b{g}$-vector fan~$\gvectorFan$ is polytopal for any finite type exchange matrix~$\B_\circ$.
As discussed in Section~\ref{subsec:acyclic}, this result was previously known for acyclic finite type~exchange matrices~\cite{HohlwegLangeThomas, Stella, PilaudStump-brickPolytope}.
The proof of this paper, although similar in spirit to that of~\cite{Stella}, actually simplifies the previous approaches.

We first consider some convenient functions which will be used later in Theorem~\ref{thm:polytopality} to lift the $\b{g}$-vector fan.
The existence of such functions will be discussed in Proposition~\ref{prop:existenceExchangeSubmodularFunctions}.

\begin{definition}
\label{def:exchangeSubmodularFunction}
A positive function~$\F$ on the cluster variables of~$\clusterAlgebra$ is \defn{exchange submodular}\footnote{The term ``exchange submodular'' is inspired from a particular situation in type~$A$ (namely, when~$\B_\circ$ has $1$'s on the upper subdiagonal and $-1$'s on the lower subdiagonal), where such functions really correspond to the classical submodular functions, \ie functions~$\F : 2^{[n]} \to \R$ such that~$\F(X) + \F(Y) > \F(X \cup Y) + \F(X \cap Y)$ for any distinct non-trivial subsets~$X,Y$ of~$[n]$.} if, for any pair of adjacent seeds~$(\B, \coefficients, \cluster)$ and~$(\B', \coefficients', \cluster')$ with~$\cluster \ssm \{x\} = \cluster' \ssm \{x'\}$, it satisfies
\[
\F(x) + \F(x') > \max \Big( \sum_{\substack{y \in \cluster \cap \cluster' \\ b_{xy} < 0}} -b_{xy} \, \F(y) \;,  \sum_{\substack{y \in \cluster \cap \cluster' \\ b_{xy} > 0}} b_{xy} \, \F(y) \Big).
\]
\end{definition}

\begin{definition}
\label{def:pointHyperplane}
Let~$\F$ be a positive function on the cluster variables of~$\principalClusterAlgebra$ we define:
\begin{enumerate}[(i)]
\item a point
\[
\bigpoint{\B_\circ}{\seed}{\F} \eqdef \sum_{x \in \seed} \F(x) \, \bigcvector{\B_\circ^\vee}{\seed^\vee}{x^\vee} \; \in V^\vee
\]
for each seed~$\seed$ of~$\principalClusterAlgebra$,
\item a halfspace~$\HS{\B_\circ}{x}{\F}$ and a hyperplane~$\Hyp{\B_\circ}{x}{\F}$ by
\begin{align*}
&
\bigHS{\B_\circ}{x}{\F} \eqdef \bigset{\b{v} \in V^\vee}{\dotprod{\biggvector{\B_\circ}{x}}{\b{v}} \le \F(x)} \\
\qquad\text{and}\qquad &
\bigHyp{\B_\circ}{x}{\F} \eqdef \bigset{\b{v} \in V^\vee}{\dotprod{\biggvector{\B_\circ}{x}}{\b{v}} = \F(x)}
\end{align*}
for each cluster variable~$x$ of~$\principalClusterAlgebra$.
\end{enumerate}
\end{definition}

The following statement is the central result of this paper.
We refer again to Proposition~\ref{prop:existenceExchangeSubmodularFunctions} for the existence of exchange submodular functions.

\begin{theorem}
\label{thm:polytopality}
For any finite type~exchange matrix~$\B_\circ$ and any exchange submodular function~$\F$, the $\b{g}$-vector fan~$\gvectorFan$ is the normal fan of the \defn{$\B_\circ$-associahedron}~$\Asso{\F} \subseteq V^\vee$ equivalently defined as
\begin{enumerate}[(i)]
\item the convex hull of the points~$\point{\B_\circ}{\seed}{\F}$ for all seeds~$\seed$ of~$\principalClusterAlgebra$, or
\item the intersection of the halfspaces~$\HS{\B_\circ}{x}{\F}$ for all cluster variables~$x$ of~$\principalClusterAlgebra$.
\end{enumerate}
\end{theorem}

\begin{proof}
Consider two adjacent seeds~$\seed = (\B, \coefficients, \cluster)$ and~$\seed' = (\B', \coefficients', \cluster')$ with~$\cluster \ssm \{x\} = \cluster' \ssm \{x'\}$.
By Lemma~\ref{lem:linearDependence}, the linear dependence between the $\b{g}$-vectors of~$\cluster \cup \cluster'$ is of the form
\[
\biggvector{\B_\circ}{x} + \biggvector{\B_\circ}{x'} = \sum_{\substack{y \in \cluster \cap \cluster' \\ \signDep{\B_\circ}{\seed}{\seed'} \, b_{xy} > 0}} \bigsignDep{\B_\circ}{\seed}{\seed'} \, b_{xy} \, \biggvector{\B_\circ}{y}
\]
for some~$\signDep{\B_\circ}{\seed}{\seed'} \in \{\pm 1\}$ (depending on the initial exchange matrix~$\B_\circ$).
However, by Definition~\ref{def:exchangeSubmodularFunction}, the function~$\F$ satisfies
\[
\F(x) + \F(x') > \max \Big( \sum_{\substack{y \in \cluster \cap \cluster' \\ b_{xy} < 0}} -b_{xy} \, \F(y),  \sum_{\substack{y \in \cluster \cap \cluster' \\ b_{xy} > 0}} b_{xy} \, \F(y) \Big) \ge \hspace{-.5cm} \sum_{\substack{y \in \cluster \cap \cluster' \\ \signDep{\B_\circ}{\seed}{\seed'} \, b_{xy} > 0}} \hspace{-.5cm} \bigsignDep{\B_\circ}{\seed}{\seed'} \, b_{xy}  \, \F(y).
\]
Applying the characterization of Proposition~\ref{prop:characterizationPolytopalFan}, we thus immediately obtain that the $\b{g}$-vector fan~$\gvectorFan$ is the normal fan of the polytope defined as the intersection of the halfspaces~$\HS{\B_\circ}{x}{\F}$ for all cluster variables~$x$ of~$\principalClusterAlgebra$.

Finally, to show the vertex description, we just need to observe that, for any seed~$\seed$ of~$\principalClusterAlgebra$, the point~$\point{\B_\circ}{\seed}{\F}$ is the intersection of the hyperplanes~$\Hyp{\B_\circ}{x}{\F}$ for all~$x \in \seed$.
Indeed, since~$\gvectors{\B_\circ}{\seed}$ and~$\cvectors{\B_\circ^\vee}{\seed^\vee}$ form dual bases by Theorem~\ref{thm:duality}, we have for any~$x \in \seed$:
\[
\dotprod{\biggvector{\B_\circ}{x}}{\bigpoint{\B_\circ}{\seed}{\F}} = \sum_{y \in \seed} \F(y) \, \dotprod{\biggvector{\B_\circ}{x}}{\bigcvector{\B_\circ^\vee}{\seed^\vee}{y^\vee}} = \sum_{y \in \seed} \F(y) \, \delta_{x = y} \; = \; \F(x).
\qedhere
\]
\end{proof}

\begin{remark}
\label{rem:differencePoints}
By definition $\Asso{\F}$ fulfills the following properties:
\begin{itemize}
\item the normal vectors are the~$\b{g}$-vectors of all cluster variables of~$\principalClusterAlgebra$,
\item for any two adjacent seeds~$\seed = (\B, \coefficients, \cluster)$ and~$\seed' = (\B', \coefficients', \cluster')$  with~$\cluster \ssm \{x\} = \cluster' \ssm \{x'\}$, the edge of $\Asso{\F}$ joining the vertex~$\point{\B_\circ}{\seed}{\F}$ to the vertex~$\point{\B_\circ}{\seed'}{\F}$ is a negative multiple of the dual $\b{c}$-vector~$\cvector{\B_\circ^\vee}{\seed^\vee}{x^\vee} = -\cvector{\B_\circ^\vee}{\seed'^\vee}{x'^\vee}$.
More precisely,
\[
\bigpoint{\B_\circ}{\seed'}{\F} - \bigpoint{\B_\circ}{\seed}{\F} = \Big(\sum_y b_{xy} \, \F(y) - h(x') - h(x) \Big) \, \bigcvector{\B_\circ^\vee}{\seed^\vee}{x^\vee},
\]
where the sum runs over the variables~$y \in \seed \cap \seed'$ such that~$b_{xy}$ has the same sign as the $\b{c}$-vector~$\cvector{\B_\circ^\vee}{\seed^\vee}{x^\vee}$.
Note that~$\sum_y b_{xy} \, \F(y) - h(x') - h(x) < 0$ since~$\F$ is exchange submodular.
\end{itemize}
\end{remark}

Our next step is to show the existence of exchange submodular functions for any finite type cluster algebra with principal coefficients.
The important observation here is that the definition of exchange submodular function does not involve in any way the coefficients of $\principalClusterAlgebra$ so that it suffices to construct one in the coefficient free cases.
Indeed, if $\F$ is exchange submodular for $\coefficientFreeClusterAlgebra$, and $\eta$ is the coefficient specialization morphism given by
\[
  \begin{array}{rccc}
    \eta: & \principalClusterAlgebra & \longrightarrow & \coefficientFreeClusterAlgebra\\
          & p_i                      & \longmapsto     & 1
  \end{array}
\]
one gets the desired map by setting
\[
  \F(x) \eqdef \F(\eta(x))
\]
for any cluster variable $x$ of $\principalClusterAlgebra$.

Recall that, up to an obvious automorphism of the ambient field, there exists a unique cluster algebra without coefficients for each given finite type~\cite{FominZelevinsky-ClusterAlgebrasIV}.
We can therefore, without loss of generality, assume that $\B_\circ$ is bipartite and directly deduce our result from \cite[Prop.~8.3]{Stella} obtained as an easy consequence of \cite[Lem.~2.4]{ChapotonFominZelevinsky} which we recast here in our current setting.

When~$\B_\circ$ is acyclic, the Weyl group of~$\A{\B_\circ}$ is finite and has a longest element~$\wo$.
A point~$\lambda^\vee \eqdef \sum_{x \in \cluster_\circ} \lambda^\vee_x \, \fundamentalWeight^\vee_x$ in the interior of the fundamental Weyl chamber of~$\A{\B^\vee_\circ}$ (that is to say $\lambda^\vee_x > 0$ for all $x \in \cluster_\circ$) is \defn{fairly balanced} if~$\wo(\lambda^\vee) = -\lambda^\vee$.

\begin{proposition}
\label{prop:existenceExchangeSubmodularFunctions}
Let $\coefficientFreeClusterAlgebra$ be any finite type cluster algebra without coefficients and assume that $\B_\circ$ is bipartite.
To each fairly balanced point~$\lambda^\vee$ corresponds an exchange submodular function $\F_{\lambda^\vee}$ on $\coefficientFreeClusterAlgebra$.
\end{proposition}

\begin{proof}
  To define $\F_{\lambda^\vee}$, recall from the construction in \cite{FominZelevinsky-YSystems, FominZelevinsky-ClusterAlgebrasII} that the set of cluster variables in $\coefficientFreeClusterAlgebra$ is acted upon by a dihedral group generated by the symbols $\tau_+$ and $\tau_-$.

  Each $\langle\tau_+,\tau_-\rangle$-orbit of cluster variables meets the initial seed in either 1 or 2 elements.
  We will define $\F_{\lambda^\vee}$ to be constant on the orbits of this action: on any element in the orbit of the initial cluster variable $x$ the function $\F_{\lambda^\vee}$ evaluates to the $x$-th coordinate of $\lambda^\vee$ when written in the basis of simple coroots.
  The requirement that $\lambda^\vee$ is fairly balanced, as explained in the proof of \cite[Thm.~6]{Stella}, is tantamount to say that $\F_{\lambda^\vee}(x)=\F_{\lambda^\vee}(y)$ if $x$ and $y$ are initial cluster variables in the same $\langle\tau_+,\tau_-\rangle$-orbit.

  The fact that the function $\F_{\lambda^\vee}$ defined in this way is exchange submodular is then the content of \cite[Lem.~2.4]{ChapotonFominZelevinsky} together with the computation following \cite[Prop.~8]{Stella}.
  The only minor thing to observe is that, there, this function appears as a piecewise linear function on the ambient space of the root lattice and thus, instead of writing the cluster expansions of the two monomials in the right hand side of the exchange relations as we do here, their total denominator vectors appear.
\end{proof}

\begin{remark}
  Reading through \cite{Stella} one may have the impression, as the author did at the time, that for a given $\lambda^\vee$ a different exchange submodular function has been constructed for each choice of acyclic initial seed.
  This is not the case.
  Indeed, if one unravels the definitions, it is easy to see that the functions defined there only differ because the set of $\b{g}$-vectors they use as domain are different.
  Taken as function on the cluster variables in the respective cluster algebras with principal coefficients, under the specialization maps $\eta$, they all correspond to the same function on $\coefficientFreeClusterAlgebra$.
\end{remark}
 
A particular example of fairly balanced point is the point
\[
\rho^\vee \eqdef \sum_{x \in \cluster_\circ} \fundamentalWeight^\vee_x.
\]
Note that~$\rho^\vee$ is both the sum of the fundamental coweights and the half sum of all positive coroots of the root system of finite type~$\A{\B_\circ}$.
In particular~$\F_{\rho^\vee}$ is the \defn{half compatibility sum} of~$x$, \ie the half sum of the compatibility degrees
\[
\F_{\rho^\vee}(x) \eqdef \frac{1}{2} \sum_{y \ne x} \compatibilityDegree{y}{x}
\]
over all cluster variables distinct from~$x$.
(See \cite{FominZelevinsky-YSystems, CeballosPilaud} for the definition and discussion of the relevant properties of compatibility degrees.)
The point $\rho^\vee$ is particularly relevant in representation theory and its role in this context has already been observed in \cite[Rem.~1.6]{ChapotonFominZelevinsky}.
We call \defn{balanced $\B_\circ$-associahedron} and denote by~$\Asso{}$ the~$\B_\circ$-associahedron~$\Asso{\F_{\rho^\vee}}$ for the exchange submodular function~$\F_{\rho^\vee}$.

\begin{example}
When~$\B_\circ$ is acyclic, the $\B_\circ$-associahedron~$\Asso{}$ was already constructed in \cite{HohlwegLangeThomas, Stella, PilaudStump-brickPolytope}.
It is then obtained by deleting inequalities from the facet description of the permutahedron of the Coxeter group of type~$\A{\B_\circ}$.
Section~\ref{subsec:acyclic} provides a detailed discussion of the acyclic case.
\end{example}

\begin{example}
\label{exm:specific2}
Following Example~\ref{exm:specific1}, we have represented in \fref{fig:associahedra} the $\B_\circ$-associahedra $\Asso{}$ for the same two initial exchange matrices
\\[.3cm]
\centerline{
\begin{tabular}{c@{\qquad}c@{\qquad}c}
	$\begin{bmatrix} 0 & -1 & 1 \\ 1 & 0 & -1 \\ -1 & 1 & 0 \end{bmatrix}$ & and & $\begin{bmatrix} 0 & -1 & 2 \\ 1 & 0 & -2 \\ -1 & 1 & 0 \end{bmatrix}$ \\[.6cm]
	(type~$A_3$ cyclic) & & (type~$C_3$ cyclic)
\end{tabular}
} \\[.2cm]
Note that the leftmost associahedron of \fref{fig:associahedra} appeared as a mysterious realization of the associahedron in~\cite{CeballosSantosZiegler}.

\begin{figure}
	\capstart
	\centerline{$\vcenter{\hbox{\includegraphics[scale=.65]{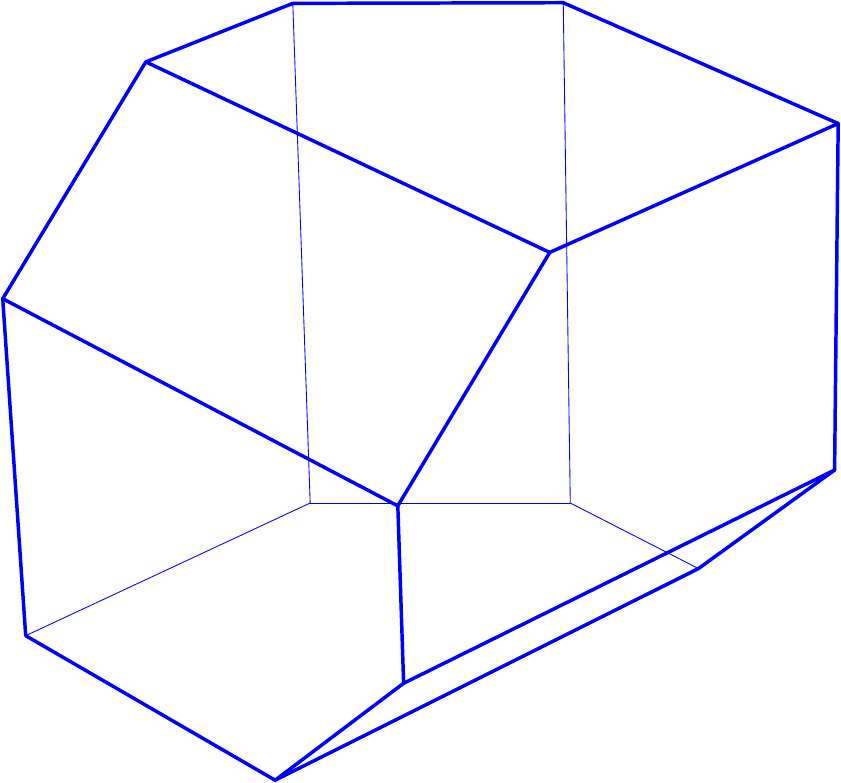}}}$ \qquad $\vcenter{\hbox{\includegraphics[scale=.7]{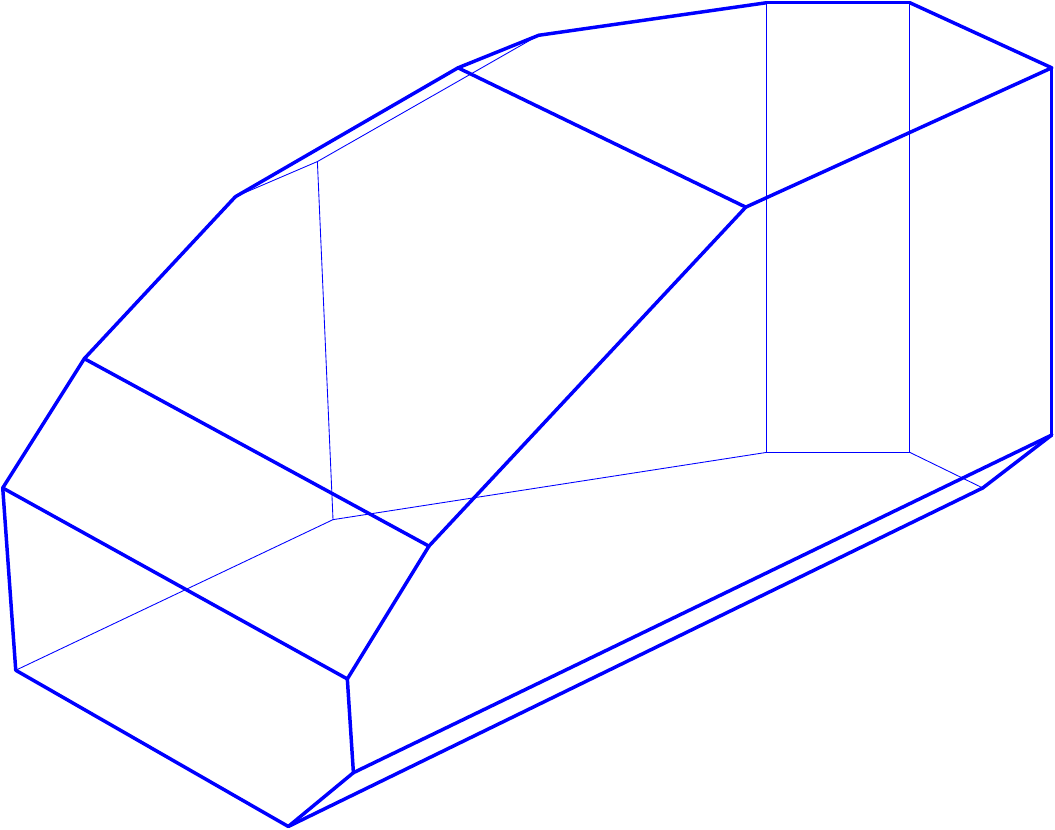}}}$}
	\caption{The associahedra~$\Asso{}$ for the type~$A_3$~(left) and type~$C_3$~(right) cyclic initial exchange matrices.}
	\label{fig:associahedra}
\end{figure}
\end{example}

We sum up by stating the main result of this paper.

\begin{corollary}
\label{coro:polytopalitygvectorFan}
For any finite type exchange matrix~$\B_\circ$, the $\b{g}$-vector fan~$\gvectorFan$ is polytopal.
\end{corollary}

%%%%%%%%%%%%%%%%%%%%%%%%%%%%%%%%%%%%%%

\section{Two families of examples}
\label{sec:specificFamilies}

Before investigating more combinatorial and geometric properties of the $\B_\circ$-associahedron $\Asso{}$, we take a moment to study two specific families of examples, corresponding to initial exchange matrices that are either acyclic (Section~\ref{subsec:acyclic}), or of type~$A$ (Section~\ref{subsec:typeA}).

%%%%%%%%

\subsection{Acyclic case}
\label{subsec:acyclic}

\enlargethispage{.4cm}
We first consider an acyclic initial seed, \ie with exchange matrix~$\B_\circ$ whose Cartan companion~$\A{\B_\circ}$ is itself a Cartan matrix of finite type.
We denote by~$W(\B_\circ)$ the Weyl group of type~$\A{\B_\circ}$ and by~$\wo$ its longest element.
Note that the choice of an acyclic seed is equivalent to the choice of a Coxeter element~$c$ of~$W(\B_\circ)$.

We gather in the following list some relevant facts from the literature (some of which were already observed earlier in the text).
We refer to~\cite{Hohlweg} for a detailed survey on these properties.

\para{Roots and weights:}
\begin{itemize}
\item the cluster variables of~$\principalClusterAlgebra$ are in bijection with the almost positive roots of~$\A{\B_\circ}$,
\item the dual $\b{c}$-vectors of~$\principalClusterAlgebra$ are all the coroots of~$\A{\B_\circ}$,
\item the $\b{g}$-vectors of~$\principalClusterAlgebra$ are some weights of~$\A{\B_\circ}$.
\end{itemize}

\para{Cambrian fan and Coxeter fan:}
\begin{itemize}
\item the dual $\b{c}$-vector fan~$\cvectorFan$ coincides with the dual Coxeter fan~$\CoxeterFan$,
\item the $\b{g}$-vector fan~$\gvectorFan$ coincides with the $c$-Cambrian fan of N.~Reading and D.~Speyer~\cite{ReadingSpeyer}.
\end{itemize}

\para{HLT associahedron and permutahedron:}
Consider any fairly balanced point~$\lambda^\vee \eqdef \sum_{x \in \cluster_\circ} \lambda^\vee_x \fundamentalWeight^\vee_x$ in the fundamental chamber of~$\A{\B^\vee_\circ}$.
Then the $\B_\circ$-associahedron~$\Asso{\F_{\lambda^\vee}\!}$ coincides with the $c$-associahedron~$\Asso[c]{\lambda^\vee\!\!}$ constructed by C.~Hohlweg, C.~Lange and H.~Thomas~\cite{HohlwegLangeThomas} and later revisited by S.~Stella~\cite{Stella} and by V.~Pilaud and C.~Stump~\cite{PilaudStump-brickPolytope} in the context of brick polytopes.
In particular, $\Asso[c]{\lambda^\vee\!\!}$ is defined by the inequalities normal to the $\b{g}$-vectors of~$\principalClusterAlgebra$ in the facet description of the \defn{$\B_\circ$-permutahedron}
\begin{align*}
\Perm{\lambda^\vee\!\!}
& \eqdef \conv \set{w(\lambda^\vee)}{w \in W(\B_\circ)} \\
& = \bigset{\b{v} \in V^\vee}{\dotprod{w(\omega_x)}{\b{v}} \le \lambda^\vee_x \text{ for all } x \in \cluster_\circ, w \in W(\B_\circ)}.
\end{align*}
See~\cite{Hohlweg} for more details on the relation between~$\Perm{\lambda^\vee\!\!}$ and~$\Asso[c]{\lambda^\vee\!\!}$.

\para{Cambrian lattice and weak order:}
When oriented in the direction~$-\sum_{x \in \cluster_\circ} \fundamentalWeight_x$,
\begin{itemize}
\item the graph of the $\B_\circ$-permutahedron is the Hasse diagram of the weak order~on~$W(\B_\circ)$,
\item the graph of the $\B_\circ$-associahedron is the Hasse diagram of the $c$-Cambrian lattice of N.~Reading~\cite{Reading-CambrianLattices}, which is a lattice quotient and a sublattice of the weak order on~$W(\B_\circ)$.
\end{itemize}

\para{Vertex barycenter:}
\enlargethispage{.3cm}
For any fairly balanced point~$\lambda^\vee$, the origin is the vertex barycenter of both the $\B_\circ$-permutahedron~$\Perm{\lambda^\vee\!\!}$ and the $\B_\circ$-associahedron~$\Asso{\lambda^\vee\!\!}$:
\[
\sum_{w \in W(\B_\circ)} w(\lambda^\vee) = \sum_{\seed} \bigpoint{\B_\circ}{\seed}{\lambda^\vee\!\!} = \b{0}.
\]
In type~$A$, this property was observed by F.~Chapoton for J.-L.~Loday's realization of the classical associahedron~\cite{Loday} and conjectured for arbitrary Coxeter element by C.~Hohlweg and C.~Lange in~\cite{HohlwegLange} in the balanced case.
It was later proved by C.~Hohlweg, J.~Lortie and A.~Raymond~\cite{HohlwegLortieRaymond} and revisited by C.~Lange and V.~Pilaud in~\cite{LangePilaud}.
Both proofs use an orbit refinement of this property.
For arbitrary finite types, it was conjectured by C.~Hohlweg, C.~Lange and H.~Thomas in~\cite{HohlwegLangeThomas} and proved by V.~Pilaud and C.~Stump using the brick polytope approach~\cite{PilaudStump-barycenter}.

\smallskip
We will see in Section~\ref{sec:properties} how these properties of the $\B_\circ$-associahedron~$\Asso{\F_{\lambda^\vee}\!}$ for finite type acyclic initial exchange matrices~$\B_\circ$ extend to arbitrary initial exchange matrices.

%%%%%%%%

\subsection{Type~$A$}
\label{subsec:typeA}

We now consider an initial exchange matrix of type~$A$.
It is well known that the vertices of the type~$A$ associahedron are counted by the Catalan numbers and are therefore in bijection with all Catalan families.
We use in this paper the classical model by triangulations of a convex polygon~\cite{FominZelevinsky-YSystems}, that we briefly recall now.

\para{Triangulation model for type~$A$ cluster algebras.}
Consider a convex $(n+3)$-gon~$\polygon_n$.
A \defn{dissection} is a set of pairwise non-crossing internal diagonals of~$\polygon_n$, and a \defn{triangulation} is a maximal dissection (thus decomposing~$\polygon_n$ into triangles).
A triangulation~$\triangulation$ defines a matrix~$\B(\triangulation) = (b_{\gamma\delta})$ whose rows and columns are indexed by the diagonals of~$\triangulation$ and where~\[
b_{\gamma\delta} = 
\begin{cases}
	1 & \text{if~$\gamma$ follows~$\delta$ in counter-clockwise order around a triangle of~$\triangulation$,} \\
	-1 & \text{if~$\gamma$ precedes~$\delta$ in counter-clockwise order around a triangle of~$\triangulation$,} \\
	0 & \text{otherwise.}
\end{cases}
\]

A \defn{flip} in a triangulation~$\triangulation$ consists in exchanging an internal diagonal~$\gamma$ by the other diagonal~$\gamma'$ of the quadrilateral formed by the two triangles of~$\triangulation$ containing it.
The reader can observe that the exchange matrix~$\B(\triangulation')$ of the resulting triangulation~$\triangulation' = \triangulation \symdif \{\gamma, \gamma'\}$ is obtained from the exchange matrix~$\B(\triangulation)$ by a mutation in direction~$\gamma$ (as defined in Section~\ref{sec:prerequisitesCA}).

Moreover, note that if the triangulation~$\triangulation$ has no internal triangle, then we can order linearly its diagonals~$\gamma_1, \dots, \gamma_n$ such that~$b_{\gamma_i, \gamma_{i+1}} = \pm 1$ and~$b_{\gamma_i, \gamma_j} = 0$ if~$|i-j| \ne 1$, so that the Cartan companion~$\A{\B(\triangulation)}$ is precisely the type~$A$ Cartan matrix.
The reciprocal statement clearly holds as well: any exchange matrix whose Cartan companion is the type~$A$ Cartan matrix is the exchange matrix~$\B(\triangulation)$ of a triangulation~$\triangulation$ with no internal triangle.
Therefore, the flip graph on triangulations of~$\polygon_n$ completely encodes the combinatorics of mutations in type~$A$.

The choice of a type~$A$ initial~exchange matrix~$\B_\circ$ is thus equivalent to the choice of an initial triangulation~$\triangulation_\circ$ of~$\polygon_n$.
The cluster algebra~$\coefficientFreeClusterAlgebra[\B(\triangulation_\circ)]$ has a cluster variable~$x_\delta$ for each (internal) diagonal of the polygon~$\polygon_n$.
Recall that if~$\delta, \delta'$ are the two diagonals of a quadrilateral with edges~$\kappa, \lambda, \mu, \nu$ in cyclic order, then the corresponding cluster variables are related by the exchange relation~$x_\delta x_{\delta'} = x_\kappa x_\mu + x_\lambda x_\nu$.
(Here and elsewhere it is understood that if $\kappa$ is a boundary edge, then $x_\kappa=1$ and similarly for $\lambda$, $\mu$, and $\nu$.)

The compatibility degree in type~$A$ is very simple: for~$\gamma \ne \delta$, we have~$\compatibilityDegree{x_\gamma}{x_\delta} = 1$ if~$\gamma$ and~$\delta$ cross, and~$0$ otherwise.
In particular, the function~$\F_\rho$ is given by~$\F_\rho(x_\gamma) = i(n-2-i)/2$ for a diagonal~$\gamma$ with~$i$ vertices of~$\polygon_n$ on one side and~$n-2-i$ on the other side.

Finally, note that since the exchange matrix~$\B(\triangulation_\circ)$ is skew-symmetric, the algebra~$\principalClusterAlgebra[\B(\triangulation_\circ)]$ and its dual~$\principalClusterAlgebra[\B(\triangulation_\circ)^\vee]$ coincide.
We therefore take the freedom to omit mentioning duals in all this type~$A$ discussion.

\para{Shear coordinates for $\b{g}$- and $\b{c}$-vectors.}
We now provide a combinatorial interpretation for the $\b{g}$- and $\b{c}$-vectors in a type~$A$ cluster algebra in terms of the triangulation model.
Our presentation is a light version (for the special case of triangulations of a disk) of the shear coordinates of~\cite{FominShapiroThurston, FominThurston} developed to provide combinatorial models for cluster algebras from surfaces.

We consider $2n+6$ points on the unit circle labeled clockwise by \mbox{$1_\circ$, $2_\bullet$, \dots, $(2n+5)_\circ$, $(2n+6)_\bullet$}.
We say that~$1_\circ, 3_\circ, \dots, (2n+5)_\circ$ are the \defn{hollow vertices} and that~$2_\bullet, 4_\bullet, \dots, (2n+6)_\bullet$ are the \defn{solid vertices}.
We simultaneously consider \defn{hollow triangulations} (based on hollow vertices) and \defn{solid triangulations} (based on solid vertices), but never mix hollow and solid vertices in our triangulations.
To help distinguishing them, hollow vertices and diagonals appear red while solid vertices and diagonals appear blue in all pictures.
See \eg \fref{fig:exmTriangulations}.

Let~$\triangulation$ be a hollow (resp.~solid) triangulation, let~$\delta \in \triangulation$, and let~$\gamma$ be a solid (resp.~hollow) diagonal.
We denote by~$\quadrilateral(\delta \in \triangulation)$ the quadrilateral formed by the two triangles of~$\triangulation$ incident to~$\delta$.
When~$\gamma$ crosses~$\delta$, we define~$\sign{\delta}{\triangulation}{\gamma}$ to be $1$, $-1$, or~$0$ depending on whether~$\gamma$ crosses~$\quadrilateral(\delta \in \triangulation)$ as a~$\ZZZ$, as a~$\SSS$, or in a corner.
If~$\gamma$ and~$\delta$ do not cross, then we set~${\sign{\delta}{\triangulation}{\gamma} = 0}$.

\begin{figure}
	\centerline{\includegraphics[scale=1]{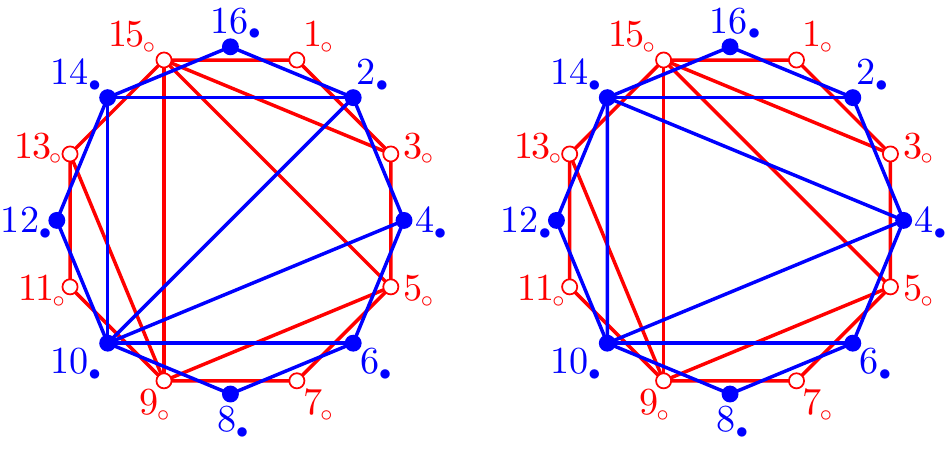}}
	\caption{Examples of hollow (red) and solid (blue) triangulations.}
	\label{fig:exmTriangulations}
\end{figure}

Fix once and for all a \defn{reference triangulation~$\triangulation_\circ$} of the hollow polygon and let~$(\fundamentalWeight_{\delta_\circ})_{\delta_\circ \in \triangulation_\circ}$ and~$(\simpleRoot_{\delta_\circ})_{\delta_\circ \in \triangulation_\circ}$ denote dual bases of~$\R^{\triangulation_\circ}$.
The reference triangulation~$\triangulation_\circ$ of the hollow polygon defines an initial triangulation~$\triangulation_\bullet^\mi \eqdef \set{(i-1)_\bullet (j-1)_\bullet}{(i_\circ, j_\circ) \in \triangulation_\circ}$ of the solid polygon, with~${\B(\triangulation_\circ) = \B(\triangulation_\bullet^\mi)}$.
The cluster algebra~$\principalClusterAlgebra[\B(\triangulation_\circ)]$ has one cluster variable~$x_{\delta_\bullet}$ for each solid internal diagonal~$\delta_\bullet$.
For a diagonal~$\delta_\bullet$ and a triangulation~$\triangulation_\bullet$ with~$\delta_\bullet \in \triangulation_\bullet$, we write~$\gvector{\triangulation_\circ}{\delta_\bullet}$ and~$\cvector{\triangulation_\circ}{\triangulation_\bullet}{\delta_\bullet}$ for the $\b{g}$- and $\b{c}$-vectors of the variable~$x_{\delta_\bullet}$ computed from the initial seed triangulation~$\triangulation_\bullet^\mi$.
We denote by~$\allcvectors{\triangulation_\circ}$ the set of all $\b{c}$-vectors of~$\principalClusterAlgebra[\B(\triangulation_\circ)]$.

The following statement provides a combinatorial interpretation of the $\b{g}$- and $\b{c}$-vectors.

\begin{proposition}
For any diagonal~$\delta_\bullet$ and a triangulation~$\triangulation_\bullet$ with~$\delta_\bullet \in \triangulation_\bullet$, we have
\[
\biggvector{\triangulation_\circ}{\delta_\bullet} \eqdef \sum_{\delta_\circ \in \triangulation_\circ} \bigsign{\delta_\circ}{\triangulation_\circ}{\delta_\bullet} \, \fundamentalWeight_{\delta_\circ}
\qquad\text{and}\qquad
\bigcvector{\triangulation_\circ}{\triangulation_\bullet}{\delta_\bullet} \eqdef \sum_{\delta_\circ \in \triangulation_\circ} -\bigsign{\delta_\bullet}{\triangulation_\bullet}{\delta_\circ} \, \simpleRoot_{\delta_\circ}.
\]
\end{proposition}

Intuitively, the $\b{g}$-vector of~$\delta_\bullet$ is given by alternating~$\pm 1$ along the \defn{zigzag} of~$\delta_\bullet$ in~$\triangulation_\circ$ (the diagonals of~$\triangulation_\circ$ that cross opposite edges of~$\quadrilateral(\delta_\circ \in \triangulation_\circ)$) and the $\b{c}$-vector of~$\delta_\bullet$~in~$\triangulation_\bullet$ is up to a sign the characteristic vector of the diagonals of~$\triangulation_\circ$ that cross opposite edges of~$\quadrilateral(\delta_\bullet \in \triangulation_\bullet)$.

For example, the solid diagonal~$2_\bullet 10_\bullet$ in the triangulation~$\triangulation^\ell_\bullet$ of \fref{fig:exmTriangulations}\,(left) has $\b{g}$-vector $\gvector{\triangulation_\circ}{2_\bullet 10_\bullet} = {\fundamentalWeight_{3_\circ 15_\circ} - \fundamentalWeight_{9_\circ 15_\circ}}$ and $\b{c}$-vector~$\cvector{\triangulation_\circ}{\triangulation_\bullet^\ell}{2_\bullet 10_\bullet} = - \simpleRoot_{5_\circ 15_\circ} - \simpleRoot_{9_\circ 15_\circ}$, while the blue diagonal~$4_\bullet 14_\bullet$ in the triangulation~$\triangulation^r_\bullet$ of \fref{fig:exmTriangulations}\,(right) has $\b{g}$-vector~$\gvector{\triangulation_\circ}{4_\bullet 14_\bullet} = \fundamentalWeight_{5_\circ 15_\circ}$ and $\b{c}$-vector~$\cvector{\triangulation_\circ}{\triangulation_\bullet^r}{4_\bullet 14_\bullet} = \simpleRoot_{5_\circ 15_\circ} + \simpleRoot_{9_\circ 15_\circ}$.

Note that there is one $\b{g}$-vector~$\gvector{\triangulation_\circ}{\delta_\bullet}$ for each internal diagonal~$\delta_\bullet$.
In contrast, many~$\delta_\bullet \in \triangulation_\bullet$ give the same $\b{c}$-vector~$\cvector{\triangulation_\circ}{\triangulation_\bullet}{\delta_\bullet}$.
For a diagonal~$\delta_\bullet = u_\bullet v_\bullet$, let~$\accordion_\circ^-$ (resp.~$\accordion_\circ^+$) denote the edges of~$\triangulation_\circ$ crossed by~$\delta_\bullet$ and not incident to the vertices~${u_\circ+1}$ or~${v_\circ+1}$ (resp.~${u_\circ-1}$ or~${v_\circ-1}$).
Define~${\b{c}^-(\delta_\bullet) \eqdef -\sum_{\delta_\circ \in \accordion_\circ^-} \omega_{\delta_\circ}}$ and~${\b{c}^+(\delta_\bullet) \eqdef \sum_{\delta_\circ \in \accordion_\circ^+} \omega_{\delta_\circ}}$.
The negative (resp.~positive) $\b{c}$-vectors of~$\allcvectors{\triangulation_\circ}$ are then precisely given by the vectors~$\b{c}^-(\delta_\bullet)$ (resp.~$\b{c}^+(\delta_\bullet)$) for all diagonals~$\delta_\bullet$ not in~$\triangulation_\bullet^\mi \eqdef \set{(i-1)_\bullet (j-1)_\bullet}{(i_\circ, j_\circ) \in \triangulation_\circ}$  (resp.~$\triangulation_\bullet^\ma \eqdef \set{(i+1)_\bullet (j+1)_\bullet}{(i_\circ, j_\circ) \in \triangulation_\circ}$).

Specializing Theorem~\ref{thm:gvectorFan}, the simplicial complex of dissections of~$\polygon_n$ is realized by the $\b{g}$-vector fan~$\gvectorFan[\triangulation_\circ] \eqdef \set{\gvector{\triangulation_\circ}{\dissection_\bullet}}{\dissection_\bullet \text{ dissection of } \polygon_n}$, which coarsens the $\b{c}$-vector fan~$\cvectorFan[\triangulation_\circ]$ (defined by the arrangement of hyperplanes orthogonal to the $\b{c}$-vectors of~$\allcvectors{\triangulation_\circ}$), which in turn coarsens the Coxeter fan~$\CoxeterFan[\triangulation_\circ]$ (defined by the Coxeter arrangement for the Cartan matrix~$\A{\B(\triangulation_\circ)}$).
These fans are illustrated in \fref{fig:exmFans} for various initial hollow triangulations~$\triangulation_\circ$.

\begin{figure}[p]
	\vspace{.5cm}
	\capstart
	\centerline{\includegraphics[scale=.4]{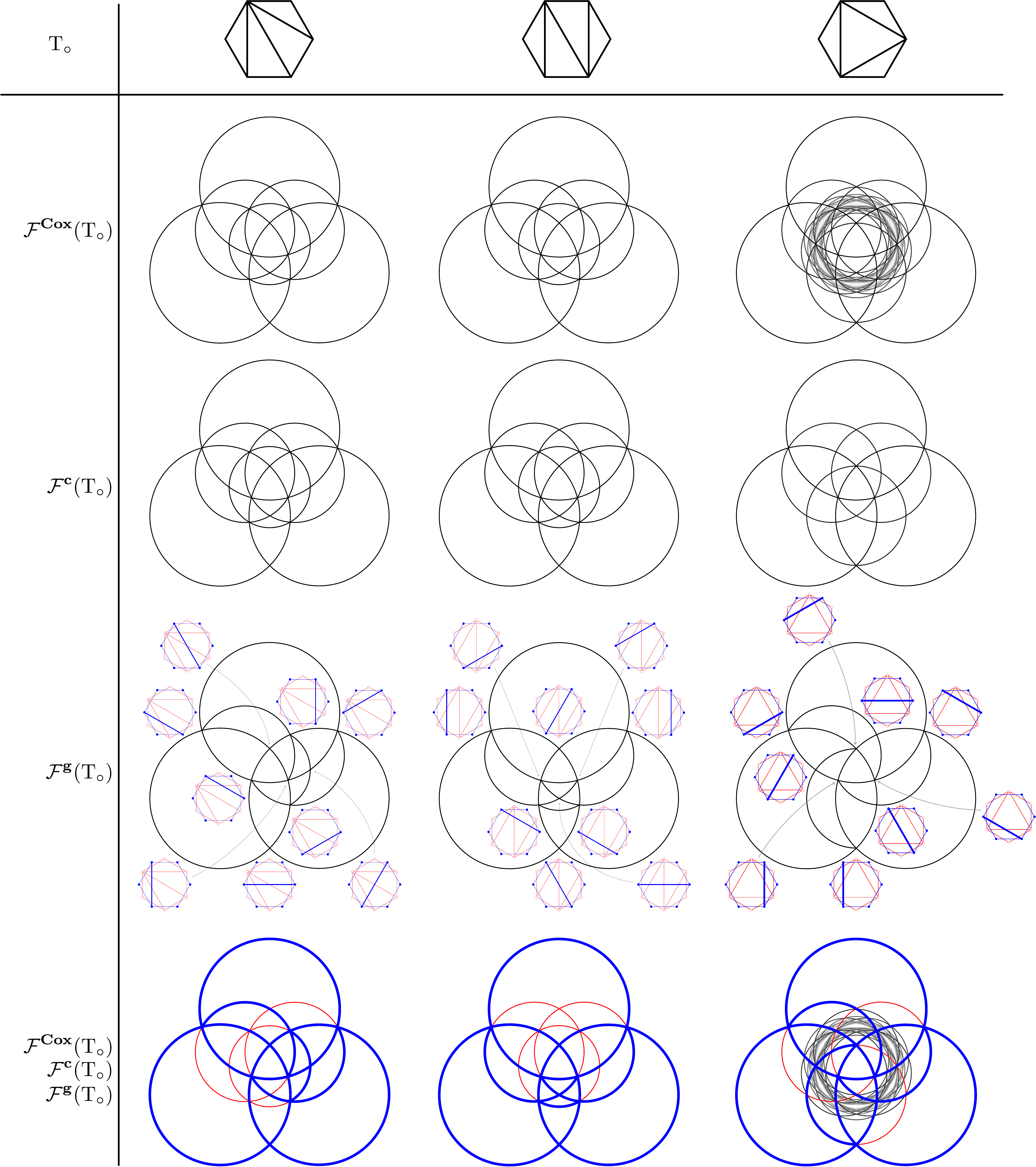}}
	\caption{Stereographic projections of the Coxeter fan~$\CoxeterFan[\triangulation_\circ]$, the $\b{c}$-vector fan~$\cvectorFan[\triangulation_\circ]$, and the $\b{g}$-vector fans~$\gvectorFan[\triangulation_\circ]$ for various reference triangulations~$\triangulation_\circ$. The $3$-dimensional fan is intersected with the unit sphere and stereographically projected to the plane from the pole in direction~$(-1,-1,-1)$.}
	\label{fig:exmFans}
\end{figure}

\para{$\triangulation_\circ$-zonotope, $\triangulation_\circ$-associahedron and $\triangulation_\circ$-parallelepiped.}
\enlargethispage{.4cm}
Using these $\b{g}$- and $\b{c}$-vectors, we now consider three polytopes associated to~$\triangulation_\circ$:
\begin{enumerate}
\item The \defn{$\triangulation_\circ$-zonotope}~$\Zono[\triangulation_\circ]{}$ is the Minkowski sum of all $\b{c}$-vectors:
\[
\Zono[\triangulation_\circ]{} \eqdef \sum_{\b{c} \in \allcvectors{\triangulation_\circ}} \b{c}.
\]
Its normal fan is the fan given by the arrangement of the hyperplanes normal to the $\b{c}$-vectors of~$\allcvectors{\triangulation_\circ}$.

\item The \defn{$\triangulation_\circ$-associahedron}~$\Asso[\triangulation_\circ]{}$ is the polytope defined equivalently as
\begin{enumerate}[(i)]
\item the convex hull of the points
\(
\bigpoint{\triangulation_\circ}{\triangulation_\bullet}{} \eqdef \sum_{\delta_\bullet \in \triangulation_\bullet} \F_\rho(\delta_\bullet) \, \cvector{\triangulation_\circ}{\triangulation_\bullet}{\delta_\bullet}
\)
for all solid triangulations~$\triangulation_\bullet$,

\item the intersection of the half-spaces
\(
{\bigHS{\triangulation_\circ}{\delta_\bullet}{} \eqdef \bigset{\b{v} \in \R^{\triangulation_\circ}\!}{\!\dotprod{\gvector{\triangulation_\circ}{\delta_\bullet}}{\b{v}} \le \F_\rho(\delta_\bullet)}}
\)
for all solid diagonal~$\delta_\bullet$.
\end{enumerate}
The normal fan of~$\Asso[\triangulation_\circ]{}$ is the $\b{g}$-vector fan~$\gvectorFan[\triangulation_\circ]$.

\item The \defn{$\triangulation_\circ$-parallelepiped}~$\Para[\triangulation_\circ]{}$ is the parallelepiped defined as
\[
\qquad\qquad
\Para[\triangulation_\circ]{} \eqdef \bigset{\b{v} \in \R^{\triangulation_\circ}}{|\dotprod{\fundamentalWeight_{\delta_\circ}}{\b{v}}| \le \F_\rho\big((i-1)_\bullet (j-1)_\bullet\big) \text{ for all } \delta_\circ = i_\circ j_\circ \in \triangulation_\circ}.
\]
Its normal fan is the coordinate fan, defined by the coordinate hyperplanes.

\end{enumerate}
These polytopes are illustrated in \fref{fig:exmAssociahedra} for various initial hollow triangulations~$\triangulation_\circ$.
Note that the $\triangulation_\circ$-zonotope~$\Zono[\triangulation_\circ]{}$ is not simple in general.
To simplify the presentation, we restricted the definition to the balanced exchange submodular function~$\F_\rho$, but similar definitions would of course hold with any exchange submodular function.

These three polytopes are strongly related: not only their normal fans, but even their inequality descriptions, refine each other, as stated in our next proposition.
This property is quite specific to type~$A$ as shown in Section~\ref{subsec:zonotope}.

\begin{figure}[p]
	\capstart
	\vspace{.5cm}
	\centerline{\includegraphics[scale=.4]{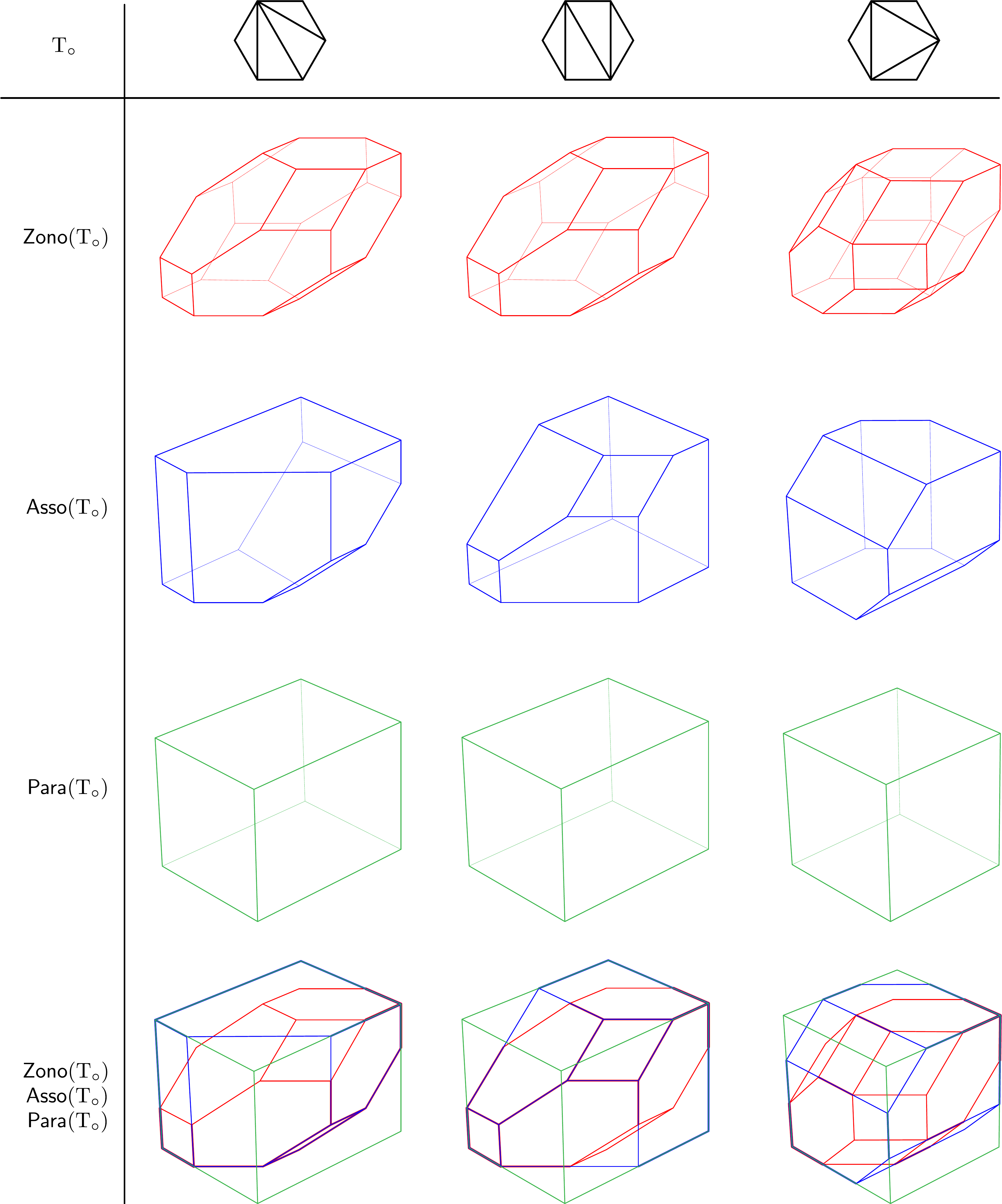}}
	\caption{The zonotope~$\Zono[\triangulation_\circ]{}$, associahedron~$\Asso[\triangulation_\circ]{}$ and parallelepiped~$\Para[\triangulation_\circ]{}$ for different reference triangulations~$\triangulation_\circ$. The first column is J.-L.~Loday's associahedron~\cite{Loday}, the second column is one of C.~Hohlweg and C.~Lange's associahedra~\cite{HohlwegLange}, the third column appeared in a discussion in C.~Ceballos, F.~Santos and G.~Ziegler's survey on \mbox{associahedra~\cite[Fig.~3]{CeballosSantosZiegler}}.}
	\label{fig:exmAssociahedra}
	\vspace{1cm}
\end{figure}

\begin{proposition}
\label{prop:removahedra}
All facet defining inequalities of~$\Para[\triangulation_\circ]{}$ are facet defining inequalities of~$\Asso[\triangulation_\circ]{}$, and all facet defining inequalities of~$\Asso[\triangulation_\circ]{}$ are facet defining inequalities of~$\Zono[\triangulation_\circ]{}$.
\end{proposition}

\begin{proof}
The first part of the sentence is immediate since for any initial diagonal~$\delta_\circ = i_\circ j_\circ$, we have
\[
\biggvector{\triangulation_\circ}{(i-1)_\bullet (j-1)_\bullet} = \fundamentalWeight_{\delta_\circ},
\qquad\text{and}\qquad
\biggvector{\triangulation_\circ}{(i+1)_\bullet (j+1)_\bullet} = -\fundamentalWeight_{\delta_\circ}.
\]
For the second part of the statement, let~$\rhs{\gamma_\bullet}$ denote the maximum of~$\dotprod{\gvector{\triangulation_\circ}{\gamma_\bullet}}{\b{v}}$ over~$\Zono[\triangulation_\circ]{}$. As $\Zono[\triangulation_\circ]{}$ is the Minkowski sum of all $\b{c}$-vectors, we have
\[
\rhs{\gamma_\bullet} = \sum_{\substack{\b{c} \in \allcvectors{\triangulation_\circ} \\ \dotprod{\gvector{\triangulation_\circ}{\gamma_\bullet}}{\b{c}} > 0}} \dotprod{\biggvector{\triangulation_\circ}{\gamma_\bullet}}{\b{c}}.
\]
To compute this sum, recall that the $\b{g}$-vector~$\gvector{\triangulation_\circ}{\gamma_\bullet}$ has alternating~$\pm 1$ along the zigzag~$\zigzag_\circ$ of~$\gamma_\bullet$ in~$\triangulation_\circ$.
Choose a $\b{c}$-vector~$\b{c} \in \allcvectors{\triangulation_\circ}$ and let~$\delta_\bullet \in \triangulation_\bullet$ be such that~${\b{c} = \cvector{\triangulation_\circ}{\triangulation_\bullet}{\delta_\bullet}}$.
Since all diagonals of~$\zigzag_\circ$ that traverse~$\quadrilateral(\delta_\bullet \in \triangulation_\bullet)$ cross it in the same way (either all as~$\SSS$ or all as~$\ZZZ$), we have~$\dotprod{\gvector{\triangulation_\circ}{\gamma_\bullet}}{\b{c}} \in \{-1,0,1\}$.
We thus want to count the $\b{c}$-vectors~$\b{c} \in \allcvectors{\triangulation_\circ}$ for which~$\dotprod{\gvector{\triangulation_\circ}{\gamma_\bullet}}{\b{c}} > 0$.
It actually turns out that it is more convenient and equivalent (since~${\allcvectors{\triangulation_\circ} = -\allcvectors{\triangulation_\circ}}$) to count the $\b{c}$-vectors~$\b{c} \in \allcvectors{\triangulation_\circ}$ for which~${\dotprod{\gvector{\triangulation_\circ}{\gamma_\bullet}}{\b{c}} < 0}$.

Decompose the zigzag~$\zigzag_\circ$ into~$\zigzag_\circ = \zigzag_\circ^+ \sqcup \zigzag_\circ^-$ such that~${\gvector{\triangulation_\circ}{\gamma_\bullet} = \sum_{\delta_\circ \in \zigzag_\circ^+} \omega_{\delta_\circ} - \sum_{\delta_\circ \in \zigzag_\circ^-} \omega_{\delta_\circ}}$.
For a diagonal~$\delta_\bullet = u_\bullet v_\bullet$, let~$\accordion_\circ^-$ (resp.~$\accordion_\circ^+$) denote the edges of~$\triangulation_\circ$ crossed by~$\delta_\bullet$ and not incident to~${u_\circ+1}$ or~${v_\circ+1}$ (resp.~${u_\circ-1}$ or~${v_\circ-1}$).
Let~${\b{c}^-(\delta_\bullet) \eqdef -\sum_{\delta_\circ \in \accordion_\circ^-} \omega_{\delta_\circ}}$ and~${\b{c}^+(\delta_\bullet) \eqdef \sum_{\delta_\circ \in \accordion_\circ^+} \omega_{\delta_\circ}}$.
Recall that the negative (resp.~positive) $\b{c}$-vectors of~$\allcvectors{\triangulation_\circ}$ are then precisely given by the vectors~$\b{c}^-(\delta_\bullet)$ (resp.~$\b{c}^+(\delta_\bullet)$) for all diagonals~$\delta_\bullet$ not in~$\triangulation_\bullet^\mi \eqdef \set{(i-1)_\bullet (j-1)_\bullet}{(i_\circ, j_\circ) \in \triangulation_\circ}$  (resp.~$\triangulation_\bullet^\ma \eqdef \set{(i+1)_\bullet (j+1)_\bullet}{(i_\circ, j_\circ) \in \triangulation_\circ}$).
We leave it to the reader to check that:
\begin{enumerate}[(i)]
\item If~$\gamma_\bullet$ and~$\delta_\bullet$ do not cross and have no common endpoint, both~$|\zigzag_\circ \cap \accordion_\circ^-|$ and~$|\zigzag_\circ \cap \accordion_\circ^+|$ are even. Thus $\dotprod{\gvector{\triangulation_\circ}{\gamma_\bullet}}{\b{c}^-(\delta_\bullet)} = \dotprod{\gvector{\triangulation_\circ}{\gamma_\bullet}}{\b{c}^+(\delta_\bullet)} = 0$.
\item If~$\gamma_\bullet$ and~$\delta_\bullet$ have a common endpoint, and~$\gamma_\bullet \delta_\bullet$ form a counterclockwise angle, then~$|{\zigzag_\circ \cap \accordion_\circ^-}|$ is even while $\zigzag_\circ \cap \accordion_\circ^+$ is empty or starts and ends in~$\zigzag_\circ^+$. Thus $\dotprod{\gvector{\triangulation_\circ}{\gamma_\bullet}}{\b{c}^-(\delta_\bullet)} = 0$ while $\dotprod{\gvector{\triangulation_\circ}{\gamma_\bullet}}{\b{c}^+(\delta_\bullet)} \ge 0$. The situation is similar if~$\gamma_\bullet \delta_\bullet$ form a clockwise angle.
\item If~$\gamma_\bullet$ and~$\delta_\bullet$ cross,~$\zigzag_\circ \cap \accordion_\circ^-$ and~${\zigzag_\circ \cap \accordion_\circ^+}$ are empty or start and end both in~$\zigzag_\circ^-$ or both in~$\zigzag_\circ^+$. Thus, either~$\dotprod{\gvector{\triangulation_\circ}{\gamma_\bullet}}{\b{c}^-(\delta_\bullet)} < 0$ and~$\dotprod{\gvector{\triangulation_\circ}{\gamma_\bullet}}{\b{c}^+(\delta_\bullet)} \ge 0$ or conversely.
\end{enumerate}
This shows that there are as many $\b{c}$-vectors~$\b{c} \in \allcvectors{\triangulation_\circ}$ for which~${\dotprod{\gvector{\triangulation_\circ}{\gamma_\bullet}}{\b{c}} < 0}$ as diagonals~$\delta_\bullet$ crossing~$\gamma_\bullet$.
In other words,~$\rhs{\gamma_\bullet} = \F_\rho(\gamma_\bullet)$.

Finally, we obtained that the inequality~$\dotprod{\gvector{\triangulation_\circ}{\gamma_\bullet}}{\b{v}} \le \rhs{\gamma_\bullet}$ defines a face~$\face(\gamma_\bullet)$ of the zonotope~$\Zono[\triangulation_\circ]{}$.
This face~$\face(\gamma_\bullet)$ is the Minkowski sum of the $\b{c}$-vectors of~$\allcvectors{\triangulation_\circ}$ orthogonal to~$\gvector{\triangulation_\circ}{\gamma_\bullet}$.
Theorem~\ref{thm:duality} ensures that any triangulation~$\triangulation_\bullet$ containing~$\gamma_\bullet$ already provides~$n-1$ linearly independent such $\b{c}$-vectors~$\cvector{\triangulation_\circ}{\triangulation_\bullet}{\delta_\bullet}$ for~${\delta_\bullet \in \triangulation_\bullet \ssm \{\gamma_\bullet\}}$.
We obtain that~$\face(\gamma_\bullet)$ has dimension~$n-1$ and is therefore a facet of the zonotope~$\Zono[\triangulation_\circ]{}$.
\end{proof}

\para{Vertex barycenter.}
We now use this (type~$A$) interpretation of the $\triangulation_\circ$-associahedron~$\Asso[\triangulation_\circ]{}$ to show that its vertex barycenter is also at the origin.
Our approach is independent, and somewhat simpler than the previous proofs of this property for type~$A$ acyclic associahedra~\cite{HohlwegLortieRaymond, LangePilaud}.

\begin{proposition}
\label{prop:barycenterTypeA}
For any initial triangulation~$\triangulation_\circ$, the origin is the vertex barycenter of the $\triangulation_\circ$-zonotope~$\Zono[\triangulation_\circ]{}$, the $\triangulation_\circ$-associahedron~$\Asso[\triangulation_\circ]{}$ and the $\triangulation_\circ$-parallelepiped~$\Para[\triangulation_\circ]{}$.
\end{proposition}

\begin{proof}
Assume that we started from a regular~$(2n+6)$-gon with alternative hollow and solid vertices.
Consider a diagonal~$\delta_\circ$ of~$\triangulation_\circ$ and let~$\Psi$ denote the reflexion of the plane which stabilizes~$\delta_\circ$.
Note that~$\Psi$ sends solid diagonals (resp.~triangulations) onto solid diagonals (resp.~triangulations).
Moreover, a diagonal~$\delta_\bullet$ crosses~$\delta_\circ$ if and only if its image~$\Psi(\delta_\bullet)$ crosses~$\delta_\circ$.
Since~$\Psi$ reverses the orientation, we therefore obtain that
\[
\bigsign{\delta_\bullet}{\triangulation_\bullet}{\delta_\circ} = - \bigsign{\Psi(\delta_\bullet)}{\Psi(\triangulation_\bullet)}{\delta_\circ},
\]
for any~$\delta_\bullet \in \triangulation_\bullet$.
Finally, $\Psi$ preserves the length of a diagonal~$\delta_\bullet$, so that~$\F_\rho(\Psi(\delta_\bullet)) = \F_\rho(\delta_\bullet)$.
Summing over all diagonals in all solid triangulations, we obtain that the $\delta_\circ$-coordinate of the vertex barycenter of~$\Asso[\triangulation]{}$ is given by
\[
\Big( \sum_{\triangulation_\bullet} \bigpoint{\triangulation_\circ}{\triangulation_\bullet}{} \Big)_{\delta_\circ}
= \sum_{\triangulation_\bullet}\sum_{\delta_\bullet \in \triangulation_\bullet} \big( \F_\rho(\delta_\bullet) \, \bigcvector{\triangulation_\circ}{\triangulation_\bullet}{\delta_\bullet} \big)_{\delta_\circ}
= - \sum_{\triangulation_\bullet}\sum_{\delta_\bullet \in \triangulation_\bullet} \F_\rho(\delta_\bullet) \, \bigsign{\delta_\bullet}{\triangulation_\bullet}{\delta_\circ}
= 0
\]
since the contribution of~$\delta_\bullet \in \triangulation_\bullet$ is balanced by that of~$\Psi(\delta_\bullet) \in \Psi(\triangulation_\bullet)$.
Since this holds for any~$\delta_\circ \in \triangulation_\circ$, we conclude that the vertex barycenter of~$\Asso[\triangulation_\circ]{}$ is the origin.
It is immediate for the other two polytopes~$\Zono[\triangulation_\circ]{}$ and~$\Para[\triangulation_\circ]{}$ since they are centrally symmetric.
\end{proof}

Proposition~\ref{prop:barycenterTypeA} will be generalized in Section~\ref{subsec:barycenter} for arbitrary seeds in arbitrary finite types and for arbitrary fairly balanced point~$\lambda$.

%%%%%%%%%%%%%%%%%%%%%%%%%%%%%%%%%%%%%%

\section{Further properties of~$\Asso{}$}
\label{sec:properties}

In this section, we discuss further geometric properties of the $\B_\circ$-associahedron~$\Asso{\F}$, motivated by the specific families of examples presented in Section~\ref{sec:specificFamilies}.
We also introduce the universal associahedron mentioned in Theorem~\ref{thm:universalAssociahedronIntro}, a high dimensional polytope which simultaneously contains the associahedra~$\Asso{\F}$ for all exchange matrices~$\B_\circ$ of a given finite type.

%%%%%%%%

\subsection{Green mutations}
\label{subsec:lattice}

Motivated by Section~\ref{subsec:acyclic}, we consider a natural orientation of mutations introduced by B.~Keller in the context of quantum dilogarithm identities~\cite{Keller}.
For two adjacent seeds~$\seed = (\B, \coefficients, \cluster)$ and~$\seed' = (\B', \coefficients', \cluster')$ of~$\principalClusterAlgebra$ with~${\cluster \ssm \{x\} = \cluster' \ssm \{x'\}}$, the mutation~${\seed \to \seed'}$ is a \defn{green mutation} when the dual $\b{c}$-vector~${\cvector{\B_\circ^\vee}{\seed^\vee}{x^\vee} = -\cvector{\B_\circ^\vee}{\seed'^\vee}{x'^\vee}}$ is positive.
The directed graph~$\greenMutationsGraph$ of green mutations is known to be acyclic, and even the Hasse diagram of a lattice when the type of~$\B_\circ$ is simply laced, see~\cite[Coro.~4.7]{GarverMcConville-OrientedExchangeGraphs} and the references therein. Further lattice theoretic properties of~$\greenMutationsGraph$ are discussed in~\cite{GarverMcConville-OrientedExchangeGraphs}.

\begin{example}
For an acyclic initial exchange matrix~$\B_\circ$, the lattice of green mutations is the $c$-Cambrian lattice of N.~Reading~\cite{Reading-CambrianLattices} (where~$c$ is the Coxeter element corresponding to~$\B_\circ$).
\end{example}

It turns out that this green mutation digraph~$\greenMutationsGraph$ is apparent in the $\B_\circ$-associahedron.
Indeed, the following statement is a direct consequence of Remark~\ref{rem:differencePoints}.

\begin{proposition}
For any finite type exchange matrix~$\B_\circ$, the graph of the associahedron~$\Asso{\F}$, oriented in the linear direction~$-\sum_{x \in \cluster_\circ} \fundamentalWeight_x$, is the graph~$\greenMutationsGraph$ of green mutations in~$\principalClusterAlgebra$.
\end{proposition}

\begin{proof}
Consider two adjacent seeds~$\seed = (\B, \coefficients, \cluster)$ and~$\seed' = (\B', \coefficients', \cluster')$ of~$\principalClusterAlgebra$ with $\cluster \ssm \{x\} = \cluster' \ssm \{x'\}$.
By Remark~\ref{rem:differencePoints}, we have
\[
\point{\B_\circ}{\seed'}{\F} - \point{\B_\circ}{\seed}{\F} = -\gamma \, \cvector{\B_\circ^\vee}{\seed^\vee}{x^\vee}
\]
for some positive~$\gamma \in \R_{>0}$.
Therefore,
\[
\bigdotprod{-\sum_{x \in \cluster_\circ} \fundamentalWeight_x}{\bigpoint{\B_\circ}{\seed'}{\F} - \bigpoint{\B_\circ}{\seed}{\F}} = \gamma \, \bigdotprod{\sum_{x \in \cluster_\circ}\fundamentalWeight_x}{\cvector{\B_\circ^\vee}{\seed^\vee}{x^\vee}}
\]
is positive if and only if~$\cvector{\B_\circ^\vee}{\seed^\vee}{x^\vee}$ is a positive $\b{c}$-vector.
\end{proof}

%%%%%%%%

\subsection{Universal associahedron}
\label{subsec:universalAssociahedron}

For each initial exchange matrix~$\B_\circ$ of a given type, we constructed in Section~\ref{sec:polytopality} a generalized associahedron~$\Asso{\F}$ by lifting the $\b{g}$-vector fan using an exchange submodular function~$\F$ on the cluster variables of~$\principalClusterAlgebra$.
As already observed though, the function~$\F$ is independent of the coefficients of~$\principalClusterAlgebra$, so that all $\b{g}$-vector fans can be lifted with the same function~$\F$.
This motivates the definition of a universal associahedron.

For this, consider the finite type cluster algebra~$\universalClusterAlgebra$ with universal coefficients, and let~$\clusterVariables$ denote its set of cluster variables.
Consider a $|\clusterVariables|$-dimensional vector space~$U$ with basis~$\{\beta_x\}_{x \in \clusterVariables}$ and its dual space~$U^\vee$ with basis~$\{\beta^\vee_{x^\vee}\}_{x^\vee \in \clusterVariables[\B_\circ^\vee]}$. As before, the cluster variables of~$\universalClusterAlgebra$ and~$\universalClusterAlgebra[\B_\circ^\vee]$ are related by~$x \leftrightarrow x^\vee$.
For~$\cluster \subseteq \clusterVariables$, we denote by~$\coordSubspace{\cluster}$ the coordinate subspace of~$U$ spanned by~$\{\beta_x\}_{x \in \cluster}$.

Given a seed~$\seed$ in~$\universalClusterAlgebra$, the \defn{$\b{u}$-vector} of a cluster variable $x \in \seed$ is the vector
\[
\uvector{\B_\circ}{\seed}{x} \eqdef \sum_{y \in \clusterVariables} u_{yx} \, \beta_y \; \in U
\]
of exponents of~$p_x = \prod_{y \in \clusterVariables} (p[y])^{u_{yx}}$.
Remark~\ref{rem:universalToPrincipal} then reformulates geometrically in terms of $\b{u}$- and $\b{c}$-vectors as follows.
Choose a seed~$\seed_\star = (\B_\star, \coefficients_\star, \cluster_\star)$ in~$\universalClusterAlgebra$ that you want to make initial.
Then, for any cluster variable~$x$ in a seed~$\seed$, the $\b{c}$-vector~$\cvector{\B_\star}{\seed}{x}$ is the orthogonal projection of the $\b{u}$-vector~$\uvector{\B_\circ}{\seed}{x}$ on the coordinate subspace~$\coordSubspace{\cluster_\star}$.
(Here and elsewhere we identify $\coordSubspace{\cluster_\star}$ with $V$ and $\coordSubspace{\cluster_\star^\vee}$ with $V^\vee$ in the obvious way.)

We are now ready to define the universal associahedron.

\begin{definition}
\label{def:universalAssociahedron}
For any finite type exchange matrix~$\B_\circ$ and any exchange submodular function~$\F$, the \defn{universal $\B_\circ$-associahedron} is the polytope~$\UnivAsso{\F}$ in~$U^\vee$ defined as the convex hull of the points
\[
\Univbigpoint{\B_\circ}{\seed}{\F} \eqdef \sum_{x \in \seed} \F(x) \, \biguvector{\B_\circ^\vee}{\seed^\vee}{x^\vee} \; \in U^\vee
\]
for each seed~$\seed$ of~$\universalClusterAlgebra$.
\end{definition}

Note that~$\UnivAsso{\F}$ does not depend on~$\B_\circ$ but only on its cluster type.
We keep~$\B_\circ$ in the notation since it fixes the indexing of the spaces~$U$ and~$U^\vee$.

\begin{example}
We illustrate Definition~\ref{def:universalAssociahedron} on the type~$C_2$ exchange matrix:
\[
\B_\circ = \begin{bmatrix} 0 & 2  \\ -1 & 0  \end{bmatrix}.
\]
The cluster algebra $\universalClusterAlgebra$ has $6$ cluster variables that we denote by $\clusterVariables = \{ x_1,x_2,x_3,x_4,x_5,x_6\}$.
It is straightforward to verify that to the point $2\rho^\vee$ corresponds the function~$\F_{2\rho^\vee}$ with value~$3$ on~$x_1, x_3, x_5$ and~$4$ on~$x_2, x_4, x_6$.
The $\b{u}$-vectors we need to compute our polytope are those associated to the algebra $\universalClusterAlgebra[\B_\circ^\vee]$ that appears in Example~\ref{exm:B2}.
(Notationally one should think of cluster variables in Example~\ref{exm:B2} as $\{x_i^\vee\}$.)
Using \fref{fig:B2_example} we then get that $\UnivAsso{\F_{2\rho^\vee}\!\!}$ is the convex hull of the $6$ points
\begin{gather*}
(3, 4, -3, -4, -1, 2), \qquad (3, -4, -3, -2, 1, 4), \qquad (1, 4, 3, -4, -3, -2), \\
(-3, -4, -1, 2, 3, 4), \qquad (-1, 2, 3, 4, -3, -4), \qquad (-3, -2, 1, 4, 3, -4).
\end{gather*}
These are in general position so $\UnivAsso{\F_{2\rho^\vee}\!\!}$ is a $5$-dimensional simplex embedded in~$\R^6$.
\end{example}

Our interest in~$\UnivAsso{\F}$ comes from the following property.

\begin{theorem}
\label{thm:universalAssociahedron}
Fix a finite type exchange matrix~$\B_\circ$ and an exchange submodular function~$\F$.
For any seed~$(\B_\star, \coefficients_\star, \cluster_\star)$ of $\universalClusterAlgebra$, the orthogonal projection of the universal associahedron~$\UnivAsso{\F}$ on the coordinate subspace~$\coordSubspace{\cluster_\star^\vee}$ of~$U^\vee$ spanned by~$\{\beta^\vee_{x^\vee}\}_{x^\vee \in \cluster^\vee_\star}$ is the $\B_\star$-associahedron~$\Asso[\B_\star]{\F}$.
\end{theorem}

\begin{proof}
Denote by~$\pi_{\cluster_\star^\vee}$ the orthogonal projection on~$\coordSubspace{\cluster_\star^\vee}$.
We already observed that
\[
\cvector{\B_\star^\vee}{\seed^\vee}{x^\vee} = \pi_{\cluster_\star^\vee} \big( \uvector{\B_\circ^\vee}{\seed^\vee}{x^\vee} \big)
\]
for any cluster variable~$x$ in any seed~$\seed$.
It follows that
\[
\bigpoint{\B_\star}{\seed}{\F} = \pi_{\cluster_\star^\vee} \big( \Univbigpoint{\B_\circ}{\seed}{\F} \big)
\]
for any seed~$\seed$.
We conclude that
\begin{align*}
\Asso{\F}
& \eqdef \conv \set{\bigpoint{\B_\star}{\seed}{\F}}{\seed \text{ seed in } \principalClusterAlgebra} \\
& = \conv \set{\pi_{\cluster_\star^\vee} \big( \Univbigpoint{\B_\star}{\seed}{\F} \big)}{\seed \text{ seed in } \universalClusterAlgebra} \\
& = \pi_{\cluster_\star^\vee} \big( \conv \set{\Univbigpoint{\B_\star}{\seed}{\F}}{\seed \text{ seed in } \universalClusterAlgebra} \big)
\defeq \pi_{\cluster_\star^\vee} \big( \UnivAsso{\F} \big).
\qedhere
\end{align*}
\end{proof}

\begin{remark}
Consider the normal fan~$\Fan$ of the universal $\B_\circ$-associahedron~$\UnivAsso{\F}$.
Then for any seed~$\seed_\star = (\B_\star, \coefficients_\star, \cluster_\star)$ in~$\universalClusterAlgebra$, the section of~$\Fan$ by the coordinate subspace~$\coordSubspace{\cluster_\star}$ of~$U$ spanned by~$\{\beta_x\}_{x \in \cluster_\star}$ is the $\b{g}$-vector fan~$\gvectorFan[\B_\star]$.
We therefore call \defn{universal $\b{g}$-vector fan} the normal fan~$\UnivgvectorFan$ of the universal $\B_\circ$-associahedron~$\UnivAsso{}$.
\end{remark}

We now gather some observations on the universal $\B_\circ$-associahedron~$\UnivAsso{\F}$:
\begin{itemize}
\item The universal associahedron~$\UnivAsso{\F}$ is \apriori{} defined in~$U^\vee$. However, computer experiments indicate that it has codimension~$1$. The linear space containing~$\UnivAsso{\F}$ seems to be expressed naturally in terms of the Cartan matrix of type~$\B_\circ$ and~$\F$.
\item As an immediate consequence of Theorem~\ref{thm:universalAssociahedron}, we obtain that the vertices of the universal associahedron~$\UnivAsso{\F}$ are precisely the points~$\Univpoint{\B_\circ}{\seed}{\F}$ for all seeds~$\seed$ of~$\universalClusterAlgebra$, and that the mutation graph of the cluster algebra~$\universalClusterAlgebra$ is a subgraph of the graph of~$\UnivAsso{\F}$. However, this inclusion is strict in general.
\item In general~$\UnivAsso{\F}$ is neither simple nor simplicial. Table~\ref{table:universalAssociahedron} presents some statistics for the number of vertices per facet and facets per vertex in the universal associahedron of type~$A_n$ for~$n \in [4]$.
\item Computer experiments indicate that the face lattice (and thus the $f$-vector) of the universal $\B_\circ$-associahedron~$\UnivAsso{\F}$ is independent of~$\F$.

\begin{table}
	\capstart
	\begin{tabular}{c|c|c|c|c|c|c}
    	$n$
		& $\parbox{\widthof{ambient space}}{\centering dimension of \\ ambient space \\[.1cm]}$ 
		& dimension 
		& \# vertices 
		& \# facets & \# vertices / facet
		& \# facets / vertex \\
		\hline
		$1$ & $2$  & $1$  & $2$  & $2$    & $1$                   & $1$                       \\
		$2$ & $5$  & $4$  & $5$  & $5$    & $4$                   & $4$                       \\
		$3$ & $9$  & $8$  & $14$ & $60$   & $9 \le \cdot \le 10$  & $30 \le \cdot \le 42$     \\
		$4$ & $14$ & $13$ & $42$ & $8960$ & $14 \le \cdot \le 28$ & $3463 \le \cdot \le 4244$
	\end{tabular}
	\caption{Some statistics for the universal associahedron of type~$A_n$ for~$n \in [4]$.}
	\label{table:universalAssociahedron}
\end{table}
\end{itemize}

To conclude, let us insist on the fact that Theorem~\ref{thm:universalAssociahedron} describes the projection of the universal associahedron~$\UnivAsso{\F}$ on coordinate subspaces corresponding to clusters of~$\universalClusterAlgebra$. It turns out that the projections on coordinate subspaces corresponding to all faces (not necessarily facets) of the cluster complex create relevant simplicial complexes, fans and polytopes~\cite{Chapoton-quadrangulations, GarverMcConville-trees, MannevillePilaud-accordion}. This naturally raises the question to understand all coordinate projections of the universal associahedron~$\UnivAsso{\F}$.

%%%%%%%%

\subsection{Vertex barycenter}
\label{subsec:barycenter}

As mentioned in Section~\ref{subsec:acyclic}, the vertex barycenters of all associahedra constructed by C.~Hohlweg, C.~Lange and H.~Thomas in~\cite{HohlwegLangeThomas} coincide with the origin.
This intriguing property observed in~\cite[Conj.~5.1]{HohlwegLangeThomas} was proved by V.~Pilaud and C.~Stump~\cite{PilaudStump-barycenter}.
We show in this section that it also extends to all associahedra~$\Asso{\F}$ for any initial exchange matrix~$\B_\circ$ and any exchange submodular function~$\F$.
In fact, it is a consequence of the following stronger statement.

\begin{theorem}
\label{thm:barycenter}
For any finite type exchange matrix~$\B_\circ$ and any exchange submodular function~$\F$, the origin is the vertex barycenter of the universal $\B_\circ$-associahedron~$\UnivAsso{\F}$.
\end{theorem}

Our proof of this theorem relies on its validity in the acyclic case~\cite{PilaudStump-barycenter}, and on the following observation.

\begin{lemma}
\label{lem:projectionBarycenter}
Fix a finite type exchange matrix~$\B_\circ$ and an exchange submodular function~$\F$.
For any seed~$(\B_\star, \coefficients_\star, \cluster_\star)$ of $\universalClusterAlgebra$, the vertex barycenter of the $\B_\star$-associahedron~$\Asso[\B_\star]{\F}$ is the image of the vertex barycenter of the universal associahedron~$\UnivAsso{\F}$ by the orthogonal projection on the coordinate subspace~$\coordSubspace{\cluster_\star^\vee}$ of~$U^\vee$ spanned by~$\{\beta^\vee_{x^\vee}\}_{x^\vee \in \cluster^\vee_\star}$.
\end{lemma}

\begin{proof}
Denote by~$\pi_{\cluster_\star^\vee}$ the orthogonal projection on~$\coordSubspace{\cluster_\star^\vee}$.
Since~$\pi_{\cluster_\star^\vee}$ is linear, we have
\[
\sum_{\seed} \bigpoint{\B_\star}{\seed}{\F} = \sum_{\seed} \pi_{\cluster_\star^\vee} \big( \Univbigpoint{\B_\circ}{\seed}{\F} \big) = \pi_{\cluster_\star^\vee} \Big( \sum_{\seed}\Univbigpoint{\B_\circ}{\seed}{\F} \Big),
\]
where~$\seed$ runs over all seeds of~$\universalClusterAlgebra$.
The result follows since these seeds index the vertices of both~$\Asso[\B_\star]{\F}$ and~$\UnivAsso{\F}$.
\end{proof}

\begin{proof}[Proof of Theorem~\ref{thm:barycenter}]
Consider a cluster variable~$x$ of~$\universalClusterAlgebra$.
Let~$(\B_\star, \coefficients_\star, \cluster_\star)$ be an acyclic seed of~$\universalClusterAlgebra$ containing~$x$ (such a seed exists, we could even require that it is bipartite).
Since by ~\cite{PilaudStump-barycenter} the vertex barycenter of $\Asso[\B_\star]{\F}$ is at the origin, we obtain by Lemma~\ref{lem:projectionBarycenter} that the $x$-coordinate of the vertex barycenter of the universal associahedron~$\UnivAsso{\F}$ vanishes.
Applying the same argument independently for all cluster variables~$x$ of~$\universalClusterAlgebra$ concludes the proof.
\end{proof}

\begin{corollary}
For any finite type exchange matrix~$\B_\circ$ and any exchange submodular function~$\F$, the origin is the vertex barycenter of the $\B_\circ$-associahedron~$\Asso{\F}$.
\end{corollary}

\begin{proof}
This is an immediate consequence of Theorem~\ref{thm:barycenter} and Lemma~\ref{lem:projectionBarycenter}.
\end{proof}

%%%%%%%%

\subsection{Zonotope}
\label{subsec:zonotope}

Motivated by the specific families presented in Section~\ref{sec:specificFamilies}, it is natural to investigate whether there exists a zonotope~$\Zono{\F}$ whose facet description contains all inequalities of the associahedron~$\Asso{\F}$.
In this section, we show that such a zonotope does not always exist in general.

\begin{figure}[b]
	\capstart
	\centerline{\includegraphics[scale=.7]{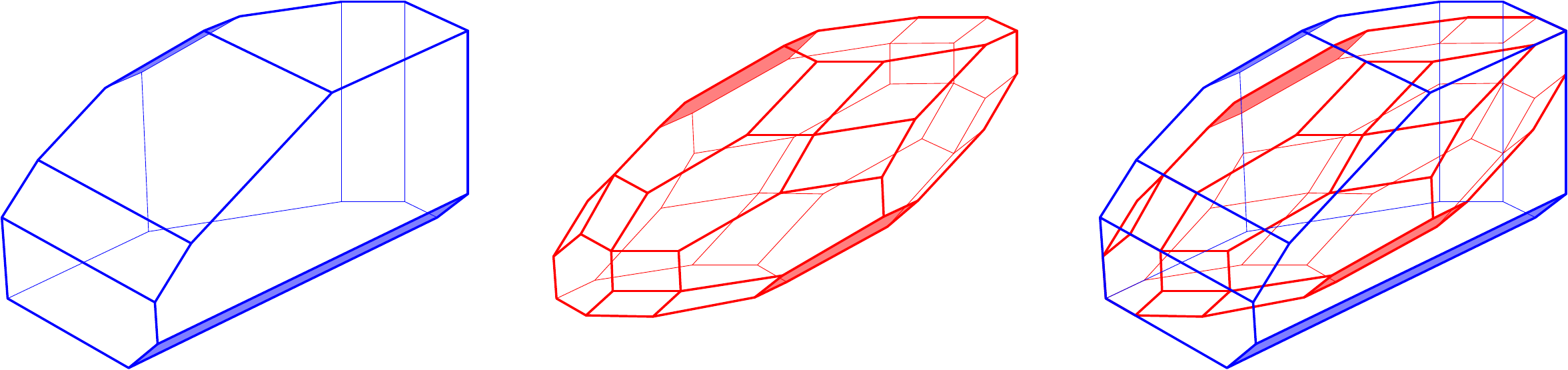}}
	\caption{The zonotope counter-example in type~$C_3$. The $\B_\circ$-associahedron~$\Asso{}$ (left), the $\B_\circ$-zonotope~$\Zono{}$ (middle), and their superposition (right) for the type~$C_3$ cyclic initial exchange matrix~$\B_\circ$. The hyperplanes supporting the shaded facets of~$\Asso{}$ are parallel to but do not coincide with the hyperplanes supporting the shaded facets of~$\Zono{}$.}
	\label{fig:C3_zono_asso_1}
\end{figure}

\para{A first naive option.}
Mimicking Section~\ref{sec:specificFamilies}, the natural choice is to consider the zonotope
\(
\Zono{} \eqdef \sum_{\b{c} \in \allcvectors{\B_\circ^\vee}} \b{c}.
\)
Indeed, we have seen in Section~\ref{sec:specificFamilies} that all inequalities of~$\Asso{}$ are inequalities of~$\Zono{}$ when~$\B_\circ$ is either acyclic or of type~$A$.

However, this property already fails for the type~$C_3$ cyclic initial exchange matrix
\[
\B_\circ = \begin{bmatrix} 0 & -1 & 2 \\ 1 & 0 & -2 \\ -1 & 1 & 0 \end{bmatrix}.
\]
One indeed checks that in the direction of the $\b{g}$-vectors~$(-1,1,0)$ and~$(1,-1,0)$, the right hand sides of the inequalities are~$4$ in~$\Asso{}$ and~$3$ in~$\Zono{}$.
This is visible in \fref{fig:C3_zono_asso_1} where the facets of~$\Asso{}$ and~$\Zono{}$ orthogonal to the $\b{g}$-vectors~$(-1,1,0)$ and~$(1,-1,0)$ are shaded.

\fref{fig:C3_zono_asso_2}\,(left) shows that the polytope defined by the inequalities of~$\Zono{}$ defining facets whose normal vectors are $\b{g}$-vectors of~$\principalClusterAlgebra$ is not an associahedron of type~$C_3$ (it is not even a simple polytope).

Note by the way that the two $\b{g}$-vectors~$(-1,1,0)$ and~$(1,-1,0)$ are opposite (thus correspond to parallel facets of~$\Asso{}$).
This should sound unusual as the only pairs of opposite $\b{g}$-vectors in both situations of Section~\ref{sec:specificFamilies} are the pairs of opposite coordinate vectors~$\{\fundamentalWeight_x, -\fundamentalWeight_x\}$ for~$x \in \cluster_\circ$.

\para{General approach.}
For any tuple~$\gamma \eqdef (\gamma_{\b{c}})_{\b{c} \in \allcvectors{\B_\circ^\vee}}$ of positive coefficients, we consider the zonotope
\[
\Zono{\gamma} \eqdef \sum_{\b{c} \in \allcvectors{\B_\circ^\vee}} \gamma_{\b{c}} \, \b{c}
\]
By definition, its normal fan is $\cvectorFan$.
For any ray~$\b{g}$ of this fan, the inequality defining the facet normal to~$\b{g}$ is given by~$\dotprod{\b{g}}{\b{v}} \le \rhs{\b{g}}$ where
\[
\rhs{\b{g}} \eqdef \sum_{\substack{\b{c} \in \allcvectors{\B_\circ^\vee} \\ \dotprod{\b{g}}{\b{c}} > 0}} \gamma_{\b{c}} \, \dotprod{\b{g}}{\b{c}}.
\]
The facet description of the zonotope~$\Zono{\gamma}$ thus contains all inequalities of the associahedron~$\Asso{\F}$ if and only if~$\gamma = \gamma(\F)$ is a positive solution to the system of linear equations
\[
\sum_{\substack{\b{c} \in \allcvectors{\B_\circ^\vee} \\ \dotprod{\gvector{\B_\circ}{x}}{\b{c}} > 0}} \gamma_{\b{c}} \, \dotprod{\biggvector{\B_\circ}{x}}{\b{c}} = \F(x)
\qquad\qquad\text{for all cluster variable~$x$ of~$\principalClusterAlgebra$.}
\]

For example, such a solution exists for the type~$C_3$ cyclic initial exchange matrix, as illustrated in \fref{fig:C3_zono_asso_2}\,(right).
Note that we had to pick different coefficients for different elements of $\allcvectors{\B_\circ^\vee}$ (leading to unnaturally narrow faces) since we already observed in \fref{fig:C3_zono_asso_1} that constant coefficients are not suitable.

\begin{figure}
	\capstart
	\centerline{\includegraphics[scale=.7]{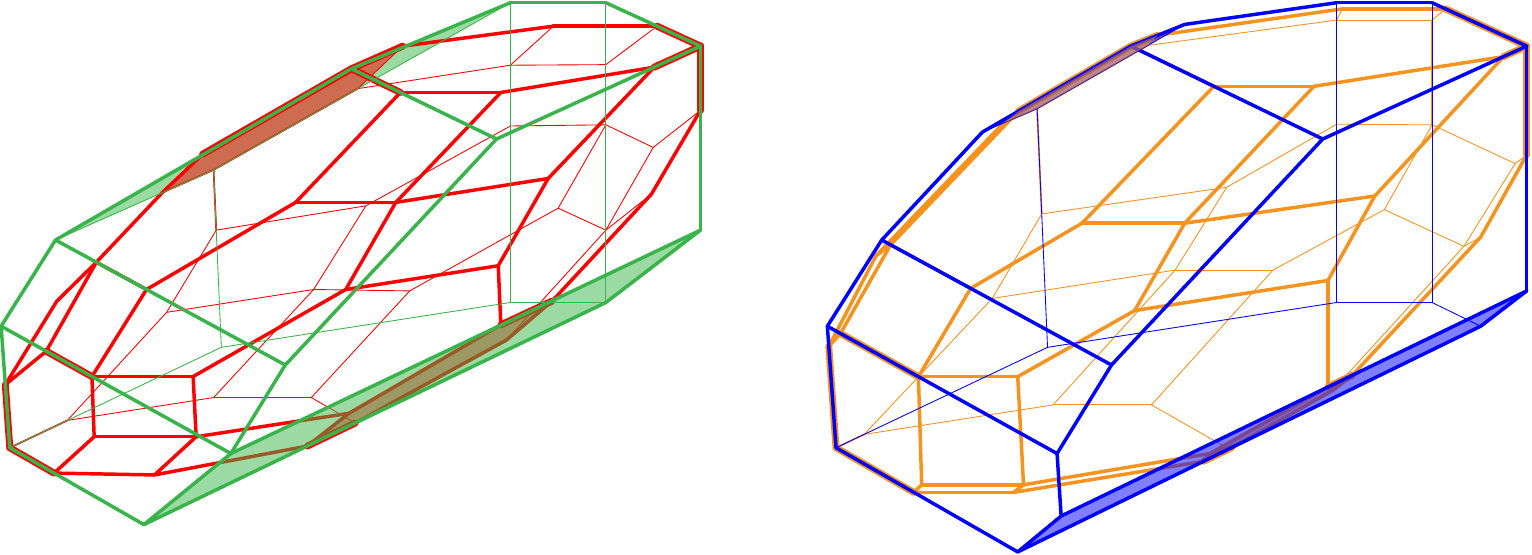}}
	\caption{The polytope defined by the inequalities of~$\Zono{}$ defining facets whose normal vectors are $\b{g}$-vectors of~$\principalClusterAlgebra$ (left), and a zonotope whose facet description contains all inequalities of~$\Asso{}$ (right).}
	\label{fig:C3_zono_asso_2}
\end{figure}

In contrast, a quick computer experiment shows that this system has no solution for the type~$D_5$ cyclic initial exchange matrix
\[
\B_\circ = \begin{bmatrix} 0 & 1 & 0 & 0 & -1 \\ -1 & 0 & 1 & 0 & 0 \\ 0 & -1 & 0 & 1 & 0 \\ 0 & 0 & -1 & 0 & 1 \\ 1 & 0 & 0 & -1 & 0 \end{bmatrix}
\]
and the balanced exchange submodular function~$\F_{\rho^\vee}$.
Rather than showing the detailed linear system, we prefer to convince the reader that there is no solution under the natural assumption that~$\gamma_{\b{c}} = \gamma_{-\b{c}}$ for any $\b{c}\in \allcvectors{\B_\circ^\vee}$, \ie that the zonotope~$\Zono{\gamma}$ is centrally symmetric.
To see that no such solution exists, consider the classical punctured pentagon model for the type~$D_5$ cluster algebra~\cite{FominShapiroThurston}.
Recall that 
\begin{enumerate}[(i)]
\item cluster variables correspond to internal arcs up to isotopy (with the subtlety that an arc incident to the puncture can be tagged or not), 
\item $\b{g}$-vectors can be read by shear coordinates~\cite{FominThurston} in a similar way as in Section~\ref{subsec:typeA}, and
\item the compatibility degree between two arcs is the (minimal) number of crossings between~them.
\end{enumerate}
We therefore obtain that the arcs~$\gamma_\ell$ connecting~$1_\bullet$ to~$3_\bullet$ in \fref{fig:exmTriangulationsD5}\,(left) and the arc~$\gamma_r$ connecting~$4_\bullet$ to~$5_\bullet$ in \fref{fig:exmTriangulationsD5}\,(right) satisfy:
\begin{gather*}
\biggvector{\B_\circ}{x_{\gamma_\ell}} = (-1,0,1,0,0)
\quad\text{and}\quad
\F_{\rho^\vee}(x_{\gamma_\ell}) = 7, \\
\biggvector{\B_\circ}{x_{\gamma_r}} = (1,0,-1,0,0)
\quad\text{and}\quad
\F_{\rho^\vee}(x_{\gamma_\ell}) = 9.
\end{gather*}
In other words, we exhibited two cluster variables with opposite $\b{g}$-vectors while belonging to distinct $\langle\tau_+,\tau_-\rangle$-orbits.
Since a centrally symmetric zonotope would have the same right-hand-sides on opposite normal vectors, this shows that no such polytope exists.

\begin{figure}
	\centerline{\includegraphics[scale=1]{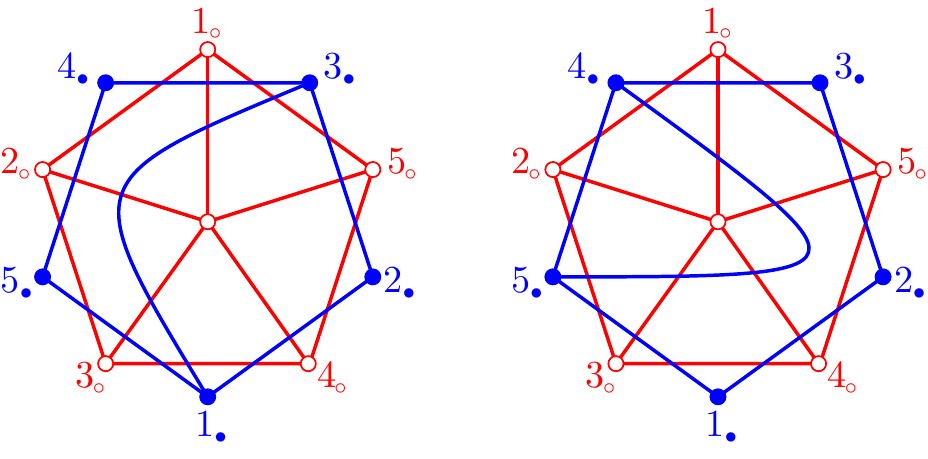}}
	\caption{An obstruction to the existence of a centrally symmetric zonotope whose facet description would contain that of the associahedron~$\Asso{}$ for a cyclic initial seed~$\B_\circ$ in type~$D_5$. The two arcs~$\gamma_\ell$ (left) and~$\gamma_r$ (right) have opposite $\b{g}$-vectors but different compatibility sums.}
	\label{fig:exmTriangulationsD5}
\end{figure}

\section*{Aknowledgements}

We thank N.~Reading for helpful discussions on this topic and comments on a preliminary version of this paper.
We are also grateful to A.~Garver and T.~McConville for bibliographic inputs.
This work answers one of the questions A. Zelevinsky posed to the third author as possible Ph.D. problems. 
He would therefore like to express once again gratitude to his advisor for the guidance received.

\bibliographystyle{alpha}
\bibliography{associahedronFromCyclicSeeds}
\label{sec:biblio}

\end{document}